\documentclass[a4paper,10pt]{amsart}

\usepackage[utf8]{inputenc} 

\usepackage{amsmath}
\usepackage{amstext}
\usepackage{amsfonts}
\usepackage{amsthm}
\usepackage{amssymb}
\usepackage{enumitem}

\usepackage{graphicx}
\usepackage{dsfont}
\usepackage{nicefrac}

\usepackage{hyperref}
\hypersetup{colorlinks}

\numberwithin{equation}{section}

\newcommand{\dualV}[2]{\langle #1, #2 \rangle_{V^*\times V}}
\newcommand{\dualVell}[2]{\langle #1, #2 \rangle_{V_{\ell}^*\times V_{\ell}}}

\newcommand{\dualX}[3][]{\langle #2,#3 \rangle_{#1^{*}\times #1}}

\newcommand{\ska}[3][]{( #2 , #3 )_{#1}}
\newcommand{\skab}[3][]{\big(#2,#3\big)_{#1}}
\newcommand{\skaB}[3][]{\Big(#2,#3\Big)_{#1}}
\newcommand{\W}{\mathcal{W}^p(0,T)}
\newcommand{\diff}[1]{\,\mathrm{d}#1}
\newcommand{\R}{\mathbb{R}}
\newcommand{\N}{\mathbb{N}}
\newcommand{\incl}{\hookrightarrow}
\newcommand{\weak}{\rightharpoonup}
\newcommand{\weaks}{\stackrel{\ast}{\rightharpoonup}}
\newcommand{\A}{\mathbf{A}}
\newcommand{\F}{\mathbf{f}}
\newcommand{\U}{\mathbf{U}}

\DeclareMathOperator*{\wlim}{w-lim}

\theoremstyle{plain}

\newtheorem{definition}{Definition}[section]
\newtheorem{theorem}[definition]{Theorem}
\newtheorem{lemma}[definition]{Lemma}

\newtheorem{assumption}{Assumption}

\theoremstyle{definition}
\newtheorem{remark}[definition]{Remark}

\newtheorem{example}[definition]{Example}

\begin{document}

\title[A variational approach to the sum splitting scheme]{A variational approach 
to the sum splitting scheme}

\author[M.~Eisenmann]{Monika Eisenmann}
\address{Monika Eisenmann\\
  Centre for Mathematical Sciences\\
  Lund University\\
  P.O.\ Box 118\\
  221 00 Lund, Sweden}
\email{monika.eisenmann@math.lth.se}

\author[E.~Hansen]{Eskil Hansen}
\address{Eskil Hansen\\
  Centre for Mathematical Sciences\\
  Lund University\\
  P.O.\ Box 118\\
  221 00 Lund, Sweden}
\email{eskil.hansen@math.lth.se}

\begin{abstract}
	Nonlinear parabolic equations are frequently encountered in applications and 
	efficient approximating techniques for their solution are of great importance. In 
	order to provide an effective scheme for the temporal approximation of such 
	equations, we present a sum splitting scheme that comes with a straight forward 
	parallelization strategy. 
  The convergence analysis is carried out in a variational framework that allows for 
  a general setting and, in particular, nontrivial temporal coefficients. 
  The aim of this work is to  illustrate the significant advantages of a variational 
  framework for operator splittings and use this to extend semigroup based theory 
  for this type of  schemes. 
\end{abstract}
\keywords{Nonlinear evolution problem, monotone operator, operator splitting, 
  convergence.}

\maketitle

\section{Introduction}\label{sec:intro}

Nonlinear parabolic equations, which we state as abstract evolution equation of 
the form 
\begin{equation}\label{1eq:PDE}
  u'(t)+A(t)u(t)=f(t),\quad t\in(0,T)\quad\text{and}\quad u(0)=u_{0},
\end{equation}
are frequently encountered in applications appearing in physics, chemistry and 
biology; see \cite{Aronsson.1996} and \cite[Section~1.3]{Vazquez.2007}. A few 
standard examples of the diffusion operator $A(t)v$ are 
\begin{equation}\label{1eq:op}
  -\nabla\cdot\bigl(\alpha(t) |\nabla v|^{p-2} \nabla v\bigr), 
  \quad
  -\Delta(\alpha(t) |v|^{p-2}v\bigr)
  \quad\text{and}\quad
  -\sum_{i=1}^{d} D_{i}\bigl(\alpha(t) |D_{i} v|^{p-2} D_{i} v\bigr).
\end{equation}
Here, the first and second operator is referred to as the $p$-Laplacian and the 
porous medium operator, respectively. 

Due to the problems' significance, effective techniques for their approximations 
become crucial. As we consider parabolic equations, for stability reasons the 
temporal approximation schemes need to be implicit. 
For equations which in addition are given in several 
spatial dimensions the resulting spatial and temporal approximation schemes 
require large scale computations. This typically demands implementations in 
parallel on a distributed hardware. One possibility to design temporal 
approximation schemes that can directly be implemented in a parallel fashion is 
to utilize operator splitting; see, e.g., \cite{HundsdorferVerwer.2003} for an 
introduction. 
Note that the solutions of nonlinear parabolic problems typically lack high-order 
spatial and temporal regularity. Thus, there is little use to consider high-order 
time integrators. 

In order to illustrate the splitting concept, consider the simplest implicit scheme, 
namely, the backward Euler method. For $N$ temporal steps, a step size
$k=T/N$ and the starting value $\U^{0}=u^{0}$ the backward Euler 
approximation $\U^{n}$ of $u(nk)$ is given by the recursion
\begin{equation*}
  \frac{1}{k}(\U^{n}-\U^{n-1})+k\A^{n}\U^{n}= 
  \F^{n},\quad n\in\{1,\ldots, N\},
\end{equation*}
where $(\A^{n})_{n}$ and $(\F^{n})_{n}$ are suitable 
approximations for $A$ and  $f$, respectively. Assuming that the nonlinear 
resolvent of $\A^n$ exists, we find the reformulation 
\begin{equation*}
  \U^{n}=(I+k\A^{n})^{-1}\bigl(\U^{n-1}+k\F^{n}),\quad
   n\in\{1,\ldots, N\}.
\end{equation*}
To implement one step of the backward Euler scheme in parallel, we split the Euler 
step into $s$ independently solvable problems. To this end, we decompose 
$\A^{n}$ and $\F^{n}$ as 
\begin{equation}\label{1eq:decomp}
  \A^{n}=\sum_{\ell=1}^{s} \A^{n}_{\ell}
  \quad\text{and}\quad
  \F^{n}=\sum_{\ell=1}^{s} \F^{n}_{\ell},
  \quad n\in\{1,\ldots, N\}.
\end{equation}
With this abstract operator splitting, one can design various temporal 
approximation schemes. Two possibilities to split one single Euler step are given 
by formally multiplying or adding the operators $(I+k\A^{n}_{\ell})^{-1}$, $\ell \in 
\{1,\dots,s\}$. 
A composition of such fractional step operators yields the Lie splitting scheme
\begin{equation*}
  (I+k\A^{n})^{-1}\approx\prod_{\ell=1}^{s}(I+k\A^{n}_{\ell})^{-1}.
\end{equation*}
Thus, we obtain $s$ possibly easier subproblems that are solved after each other. 
For a straightforward parallelization it is more convenient to choose a splitting, 
where the single steps can be computed at the same time. The sum splitting
\begin{equation*}
  (I+k\A^{n})^{-1}\approx \frac{1}{s} \sum_{\ell=1}^{s}(I+k s 
  \A^{n}_{\ell})^{-1}
\end{equation*}
offers this crucial advantage. The $s$ fractional steps are solved at the same time 
and their average is used as an approximation. 

The decomposition~\eqref{1eq:decomp} can be utilized in many different ways. A 
first possible application is a source term splitting, where the high-order terms 
are split from the low-order terms. For example a source term splitting of the 
reaction-diffusion equation governed by $A(t)v= 
-\nabla\cdot\bigl(\alpha(t)\nabla v\bigr) +p(t,v)$ would have the form 
\begin{equation*}
  \A_{1}^{n}v= -\nabla\cdot\bigl(\alpha(nk)\nabla v\bigr) 
  \quad\text{and}\quad 
  \A_{2}^{n}v=p(nk,v).
\end{equation*}
Here, the actions of $(I+k\A^{n}_{1})^{-1}$ can be evaluated by a standard 
fast linear elliptic solver and the actions of the nonlinear resolvent 
$(I+k\A^{n}_{2})^{-1}$ can often be expressed in a closed form. Examples 
of studies dealing with various source term splittings can be found in 
\cite{ArrarasEtAl.2017,HansenStillfjord.2013, 
  KochEtAl.2013, Eisenmann.2019}. 

Another possibility is a dimension splitting, where each spatial derivative is 
considered as a separate differential operator. For example, the dimension 
splitting of the nonlinear porous medium operator and the third operator in 
\eqref{1eq:op} are given by
\begin{equation*}
  \A_{\ell}^{n}v=-D_{\ell\ell}\bigl(\alpha(nk) |v|^{p-2}v\bigr)
  \quad\text{and}\quad
  \A_{\ell}^{n}v=-D_{\ell}\bigl(\alpha(nk) |D_{\ell} v|^{p-2} D_{\ell} v\bigr),
\end{equation*}
respectively. This splitting yields that the action of each nonlinear resolvent 
$(I+k\A^{n}_{\ell})^{-1}$
can be separated into lower-dimensional subproblems that can be solved on their 
own. Note that the $p$-Laplacian lacks a natural dimension splitting. Examples of 
convergence results for the dimension splitting of the third equation in 
\eqref{1eq:op} can be found in \cite{Temam.1968}, where the Lie 
scheme is used, and in \cite{HansenOstermann.2008}, where the sum, Lie and 
Peaceman--Rachford schemes are considered for the autonomous case.

A limitation of the dimension splitting approach is the rather large need of 
communication between the subproblems, which can impede an effective 
distributed implementation. Dimension splitting is also quite restrictive in terms 
of the spatial domains that can be considered.
A modern alternative to dimension splitting, which is applicable to a very general 
class of spatial domains, is the domain decomposition based splitting. 
Here, the subproblems are given on $s$ spatial subdomains that share a small 
overlap. As an example consider the three nonlinear diffusion 
operators~\eqref{1eq:op} and introduce 
a partition of unity $(\chi_{\ell})_{\ell=1}^{s}$, where each weight function 
$\chi_{\ell}$ vanishes outside its corresponding spatial subdomain. The domain 
decompositions $\A_{\ell}^{n}v$ are then 
\begin{align*}
  -\nabla\cdot\bigl(\chi_{\ell}\alpha(nk) |\nabla v|^{p-2} \nabla v\bigr),
  \quad
  -\Delta\bigl(\chi_{\ell}\alpha(nk) |v|^{p-2} v\bigr)\\
  \quad \text{and} \quad
  -\sum_{i=1}^{d}D_{i}\bigl(\chi_{\ell}\alpha(nk) |D_{i} v|^{p-2} D_{i} v\bigr),
\end{align*}
respectively. This approach is well suited for a parallel computation, as the 
actions of $(I+k\A^{n}_{\ell})^{-1}$ can be solved independently of each 
other and the communication required is small, due to the small overlaps between 
the subdomains. Studies regarding domain decomposition based splittings 
applied to linear and autonomous parabolic equations include 
\cite{ArrarasEtAl.2017, HansenHenningsson.2016, Mathew.1998, 
  Vabishchevich.2008}. Convergence for the Lie and sum splittings are given in 
\cite{EisenmannHansen.2018} for the autonomous $p$-Laplace and porous 
medium equations. 

Operator splitting schemes are typically analyzed in a semigroup framework, 
which yields convergence for a wide range of temporal 
approximation schemes, including the Lie and sum schemes; see 
\cite{Barbu.1976} for more details on the solution concept. However, there does 
not seem to be a straightforward way to extend the semigroup based convergence 
analysis to nonautonomous evolution equations. 
Furthermore, the semigroup framework requires some additional regularity 
conditions to relate the intersection of the domains $D(A_{\ell})$, $\ell \in 
\{1,\dots,s\}$, with the domain $D(A)$ of the full operator. The latter, e.g., implies 
restrictions on the domain decomposition of the $p$-Laplace equation 
\cite[Section~6]{EisenmannHansen.2018}.

In a variational setting this problem is avoided in a natural way while at 
the same time the analysis of nonautonomous problems is accessible. Also the 
structure of this approach is well suited to include a Galerkin scheme and 
therefore, in particular, the finite element method. 
However, the analysis typically needs to be tailored for each method.
The variational setting is a standard tool for existence theories 
\cite{Emmrich.2004, Roubicek.2013, Zeidler.1990B} and has been used 
in several works in the context of ``unsplit'' time integrators 
\cite{Emmrich.2009b, Emmrich.2009, Emmrich.2009c,    
EmmrichThalhammer.2010}. However, in the context of temporal splitting 
schemes for nonlinear parabolic equations the only variational studies that we are 
aware of is \cite{Temam.1968}. 
Here, a variational analysis is employed when proving the convergence of the Lie 
scheme applied to nonautonomous evolution equations and, as already stated, is 
applied to the dimension splitting of the third equation in \eqref{1eq:op}. 

Hence, the aim of this paper is threefold. 
Firstly, we aim to generalize the previous semigroup based analysis for the sum 
scheme to nonautonomous evolution equations without any implicit regularity 
assumptions. The latter generalization will be applicable to splittings of 
reaction-diffusion, dimension and domain decomposition type. 
Secondly, we intend to extend the abstract variational convergence results for the 
Lie scheme to the sum splitting scheme. As this requires a tailored convergence 
proof, it is not a trivial implication.
Thirdly, we also strive to illustrate the advantages of a variational approach in the 
context of splitting analyses. 

This paper is organized as follows: In Section~\ref{sec:Setting}, we state the exact 
assumptions that are needed on the abstract variational framework considered in 
the paper. This section also contains an example that shows that the relevant 
application of domain decomposition integrators for the $p$-Laplacian operator 
fits into our abstract framework. This in mind, we prove the well-posedness of the 
sum scheme, as well as suitable a priori bounds in Section~\ref{sec:solvability}. 
The main convergence results are proven in Section~\ref{sec:convergence}; see 
Theorem~\ref{thm:limitWeak} and Theorem~\ref{thm:limitStrong}.

\section{Abstract setting} \label{sec:Setting}

In this section, we introduce an abstract setting for the convergence 
analysis of the sum splitting scheme. We begin by presenting the exact 
assumptions made on the data and present the temporal discretization of 
the problem. This at hand, we can state the scheme that we will work with in this 
paper. The section ends with a more concrete setting that exemplifies the abstract 
framework.

\begin{assumption}\label{ass:spaces}
	Let $(H,\ska[H]{\cdot}{\cdot},\|\cdot\|_H)$ be a real, separable Hilbert space 
	and let $(V, \|\cdot\|_V)$ be a real, separable, reflexive Banach space such 
	that $V$ is continuously and densely embedded into $H$.  
	Further, there exist a seminorm $|\cdot|_V$ on $V$ and $c_V \in 
	(0,\infty)$ such that $\|\cdot\|_V \leq c_V \big( \|\cdot \|_H + 
	|\cdot|_V\big)$ is fulfilled.
	
	Furthermore, for $s \in \N$ let $(V_{\ell}, \|\cdot\|_{V_{\ell}})$, $\ell \in \{1,
	\dots, s\}$, be real reflexive Banach spaces that are continuously and
	densely embedded into $H$, fulfill $\bigcap_{\ell = 1}^s V_{\ell} = V$ 
	and $\sum_{\ell = 1}^{s} \|\cdot\|_{V_{\ell}}$ is equivalent to $\|\cdot\|_V$.
	For every $\ell \in \{1,\dots,s\}$, there exists a seminorm 
	$|\cdot|_{V_{\ell}}$ and $c_{V_{\ell}} \in  (0,\infty)$ such that 
	$\|\cdot\|_{V_{\ell}} \leq c_{V_{\ell}} \big( \|\cdot \|_H + 
	|\cdot|_{V_{\ell}}\big)$ and $\sum_{\ell  = 1}^{s}|\cdot|_{V_{\ell}}$ is 
	equivalent to $|\cdot|_{V}$.
\end{assumption}

Identifying $H$ with its dual space $H^*$, we obtain the Gelfand triples
\begin{align*}
	V \overset{d}{\incl} H \cong H^* \overset{d}{\incl}  V^*
	\quad \text{and}\quad
	V_{\ell} \overset{d}{\incl} H \cong H^* \overset{d}{\incl} V^*_{\ell},
	\quad \ell \in \{1,\dots,s\}.
\end{align*}
The next assumption states the properties of the differential operator that 
are of importance.

\begin{assumption}\label{ass:A}
	Let $H$ and $V$ be given as stated in Assumption~\ref{ass:spaces}.
	Furthermore, for $T>0$ and $p > 1$, let $\{A(t)\}_{t \in [0,T]}$ be a 
	family of operators such that $A(t)  \colon V \to V^*$ satisfy the 
	following conditions:
	\begin{itemize}
		\item[(1)] The mapping $Av \colon [0,T] \to V^*$, $v \in V$, given by $t 
		\mapsto A(t)v$ is continuous.
		\item[(2)] The operator $A(t) \colon V \to V^*$, $t \in [0,T]$, is
		radially continuous, i.e., the mapping $\tau \mapsto
		\dualX[V]{A(t)(u+\tau v)}{w}$ is continuous on $[0,1]$ for $u,v,w\in V$.
		\item[(3)] The operator $A(t) \colon V \to V^*$, $t \in [0,T]$, fulfills a 
		monotonicity condition such that  there exists $\eta \geq 0$ with
		\begin{align*}
			\dualX[V]{A(t)v - A(t)w}{v - w} \geq \eta |v-w|_V^p, \quad v,w \in V.
		\end{align*}
		\item[(4)] The operator $A(t) \colon V \to V^*$, $t \in [0,T]$, is uniformly
		bounded such that there exists $\beta >0$ with
		\begin{align*}
			\|A(t) v \|_{V^*} \leq \beta \big(1 + \|v\|_V^{p-1}\big), \quad v \in V.
		\end{align*}
		\item[(5)] The operator $A(t) \colon V \to V^*$, $t \in [0,T]$, fulfills a 
		coercivity condition such that there exist $\mu >0$ and $\lambda \geq 0$ 
		with
		\begin{align*}
			\dualX[V]{A(t) v}{v} + \lambda \geq  \mu |v|_{V}^p, \quad v \in V.
		\end{align*}
	\end{itemize}
\end{assumption}

Now, we can combine Assumption~\ref{ass:spaces} and Assumption~\ref{ass:A} to state 
a decomposition of the operator family $\{A(t)\}_{t \in [0,T]}$ that we employ in 
the analysis of the sum splitting scheme.

\begin{assumption}\label{ass:Aell}
	For $s \in \N$ let $H$, $V$ and $V_{\ell}$, $\ell \in \{1,\dots,s\}$,
	fulfill Assumption~\ref{ass:spaces}. For $p > 1$ and $T >0$  let the 
	operator family $\{A(t)\}_{t\in [0,T]}$ be given such that it fulfills
	Assumption~\ref{ass:A}. 
	Further, let $\{A_{\ell}(t)\}_{t\in [0,T]}$, $\ell \in \{1,\dots,s\}$, be given 
	such that $A_{\ell}(t) \colon V_{\ell} \to  V_{\ell}^*$ fulfills  
	Assumption~\ref{ass:A}, with $V$ replaced by $V_{\ell}$ for every $\ell \in 
	\{1,\dots,s\}$. Moreover, let the sum property
	\begin{align*}
		\sum_{\ell = 1}^{s} A_{\ell}(t)v = A(t)v \quad \text{in } V^*, \quad t \in 
		[0,T], v \in V
	\end{align*}
	be fulfilled.
\end{assumption}

\begin{remark}\label{rem:coefA}
	Note that the optimal coefficients $\beta, \eta, \lambda, \mu$ for the operator 
	families $\{A(t)\}_{t \in [0,T]}$, $\{A_{\ell}(t)\}_{t\in [0,T]}$, $\ell \in \{1,\dots, 
	s\}$, of operators do not necessarily have to be the same. For the sake of 
	simplicity, we assume that these coefficients coincide.
\end{remark}

We also consider the differential operators of Assumption~\ref{ass:Aell} as 
Nemytskii operators acting on spaces of Bochner integrable functions. For an 
introduction to Bochner integrable functions we refer the reader to 
\cite[Chapter~II]{DiestelUhl.1977} or \cite[Section~4.2]{Winkert.2018}. Some 
properties of such Nemytskii operators are collected in the next lemma. The 
proofs can be found in \cite[Lemma~8.4.4]{Emmrich.2004}.

\begin{lemma} \label{lemma:ABochner}
	For $p > 1$, $q = \frac{p}{p-1}$ and $T>0$ let $\{A(t)\}_{t\in[0,T] }$ 
	fulfill Assumption~\ref{ass:A}.
	Then the operator $(A v)(t) = A(t)v(t)$ maps $L^p(0,T;V)$ into $L^q(0,T;V^*)$.
	The operator is radially continuous, i.e., the mapping $\tau \mapsto \langle 
	A(u+\tau v), w \rangle_{L^q(0,T;V^*) \times L^p(0,T;V)}$
	is continuous on $[0,1]$ for all $u,v,w \in L^p(0,T;V)$.  Furthermore, it 
	fulfills a monotonicity, a boundedness and a coercivity condition such that
	\begin{align*}
		&\langle A v - A w, v -w \rangle_{L^q(0,T;V^*) \times L^p(0,T;V)}
		\geq \eta \int_{0}^{T} |v(t) - w(t) |_{V}^p \diff{t},\\
		&\|A v \|_{L^q(0,T;V^*)}
		\leq \beta \big(T^{\frac{1}{q}} + \|v\|_{L^p(0,T;V)}^{p-1}\big),\\
		&\langle A v, v \rangle_{L^q(0,T;V^*) \times L^p(0,T;V)} + \lambda T
		\geq \mu \int_{0}^{T} |v(t)|_{V}^p \diff{t}
	\end{align*}
	for all $v,w \in L^p(0,T;V)$.
\end{lemma}

The Nemytskii operator of $\{A_{\ell}(t)\}_{t\in[0,T]}$, $\ell \in 
\{1,\dots,s\}$, as introduced in Assumption~\ref{ass:Aell} also fulfills the same 
bounds with $V$ replaced by $V_{\ell}$. To make our setting complete, it remains to 
state the assumptions on $f$.

\begin{assumption}\label{ass:fell}
	Let $V$ and $V_{\ell}$, $\ell \in \{1,\dots,s\}$, fulfill
	Assumption~\ref{ass:spaces}. Let $p$ be the same value as in 
	Assumption~{\ref{ass:A}}  and $q = \frac{p}{p-1}$. Further, let $f$ be in 
	$L^{q}(0,T;V^*)$. Assume that there exist functions $f_{\ell} \in 
	L^{q}(0,T;V^*_{\ell})$, $\ell \in \{1,\dots,s\}$, such that
	\begin{align*}
		\sum_{\ell = 1}^{s}f_{\ell}(t)= f(t) \text{ in } V^*
		\quad \text{and} \quad
		\|f_{\ell} (t) \|_{V_{\ell}^*} \leq \| f(t) \|_{V^*}, \quad \text{ a.e. } t \in 
		(0,T).
	\end{align*}
\end{assumption}
Note that this assumption can be generalized to functions $f \in L^{q}(0,T;V^*) + 
L^2(0,T; H)$, compare, for example, \cite{Emmrich.2009, 
EmmrichThalhammer.2010}. In 
order to keep the presentation more simple, we only consider the smaller space 
$L^{q}(0,T;V^*)$.

We can now state the abstract evolution equation that we want to consider. 
In the following, let $\{A(t)\}_{t \in [0,T]}$ be as stated in 
Assumption~\ref{ass:Aell}, let $f$ fulfill Assumption~\ref{ass:fell} and let $u_0 
\in H$ be given. It is our overall goal to find an approximation to the solution 
$u$ of
\begin{align}
	\begin{split} \label{2eq:pde}
		\begin{cases}
			u' + A u = f \quad &\text{in } L^q(0,T;V^*),\\
			u(0) = u_0 &\text{in }H.
		\end{cases}
	\end{split}
\end{align}
This evolution equation is uniquely solvable in a variational sense with a
solution $u$ in $\W \incl C([0,T];H)$, where
\begin{align*}
	\W = \{ v \in L^p(0,T;V) : v' \in L^{q}(0,T;V^*) \};
\end{align*}
see \cite[Section~2.7]{LionsStrauss.1965} and 
\cite[Chapter~7--8]{Roubicek.2013} for further details. In the following 
analysis, we employ the sum splitting in 
order to obtain a temporal discretization of \eqref{2eq:pde}. To this end, we
consider an equidistant grid on $[0,T]$, where $N \in \N$, $k = \frac{T}{N}$ 
and $t_n = n k$ for $n \in \{ 0, \dots, N\}$. For $\ell \in \{1, \dots,s\}$ and $n 
\in \{1,\dots,N\}$ we introduce
\begin{align}\label{2eq:ApproxAF}
	\A_{\ell}^n = A_{\ell}(t_n)
	\quad \text{and} \quad
	\F_{\ell}^n = \frac{1}{k} \int_{t_{n-1}}^{t_n} f_{\ell}(t) \diff{t}.
\end{align}
We use this to construct an approximation $\U^n \approx u(t_n)$ of the
solution $u$ of \eqref{2eq:pde} for $n\in \{0,\dots,N \}$.
This approximation  is given through a recursion
\begin{align}\label{2eq:semiDiscret}
	\begin{cases}
		\frac{\U_{\ell}^n - \U^{n-1}}{k} + s \A_{\ell}^n \U_{\ell}^n = s 
		\F_{\ell}^n \quad
		&\text{in } V_{\ell}^*,
		\quad \ell \in \{ 1,\dots,s\},\\
		\U^n= \frac{1}{s} \sum_{\ell = 1}^{s} \U_{\ell}^{n} \quad &\text{in } H
	\end{cases}
\end{align}
for $n \in \{1,\dots,N\}$ with $\U^0 = u_0$.

\begin{example} \label{ex:pLap}
	 \begin{figure}
		\centering 
		\includegraphics[trim=4cm 17.5cm 7.5cm 5cm,clip, scale=0.6]{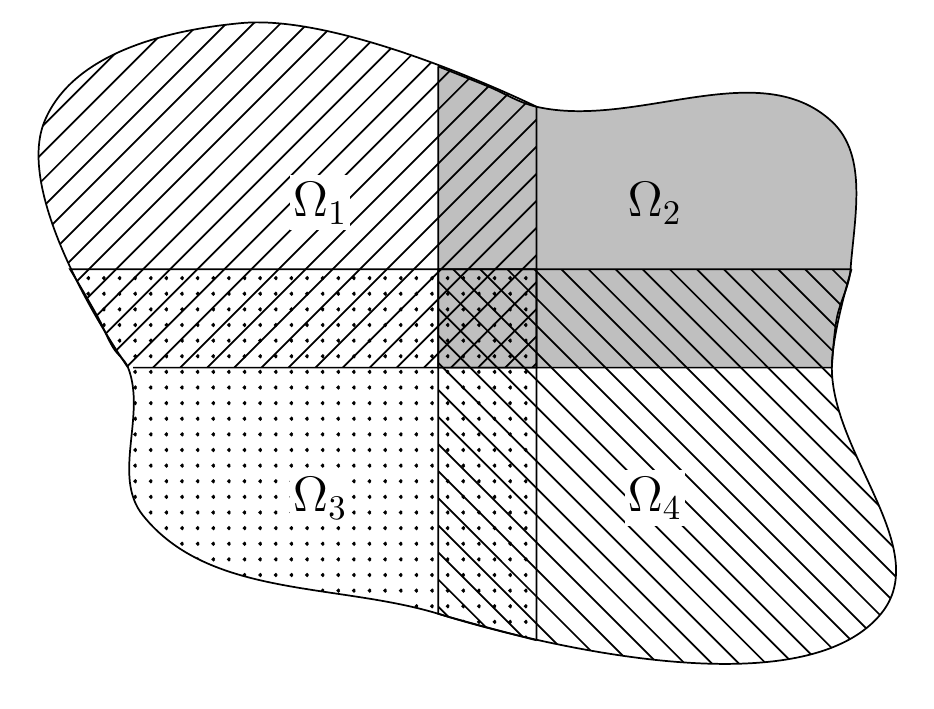}
		\includegraphics[trim=4cm 17.5cm 7.5cm 5cm,clip, scale=0.6]{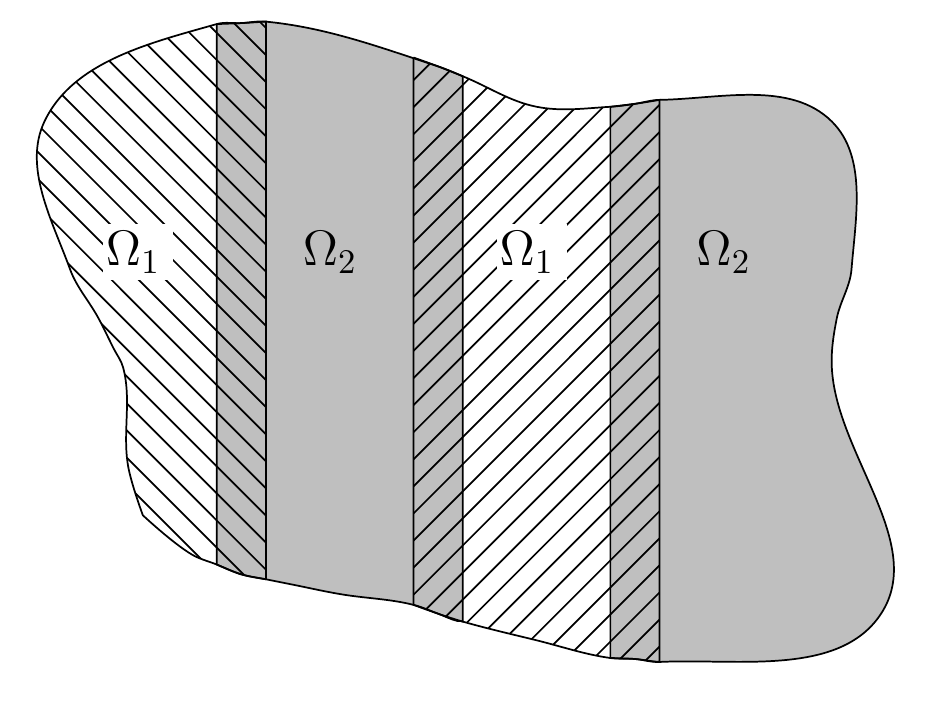}
		\caption{Examples of overlapping domains $\{ \Omega_{\ell} \}_{\ell =1}^{s}$ 
			of $\Omega\subset \R^{2}$, with $s=4$ subdomains (left) and $s=2$ 
			subdomains that are further decomposed into disjoint sets (right).}
		\label{fig:dom}
	\end{figure}
  A useful example that fits into our abstract setting is to approximate the 
  solution of the parabolic $p$-Laplace equation.
  Let $\Omega \subset \R^d$, $d\geq 1$, be given, where $\Omega$ is a bounded
  domain and the boundary $\partial \Omega$ is Lipschitz. For $p \geq 2$ we 
  consider the problem
  \begin{align}\label{2eq:parabolicPDE}
    \begin{cases}
      u_t(t,x) - \nabla \cdot (\alpha(t) |\nabla u(t,x)|^{p-2}\nabla u(t,x))  = 
      g(t,x),
      \quad &(t,x) \in (0,T) \times \Omega,\\
      \alpha(t) |\nabla u(t,x)|^{p-2} \nabla u(t,x) \cdot n = 0,
      &(t,x) \in (0,T) \times \partial \Omega,\\
      u(0,x) = u_0(x), &x \in \Omega,
    \end{cases}
  \end{align}
  where $n$ denotes outer pointing normal vector. The function $\alpha\colon 
  [0,T] \to \R$ is an element of $C([0,T])$, $u_0 \in L^2(\Omega)$ and $g \colon 
  (0,T) \times  \Omega \to \R$ is a suitably chosen integrable function that we 
  explain in more detail at a later point. 
  Applications for this type of equation can be found in \cite{Aronsson.1996}.
  Our theory allows to solve \eqref{2eq:parabolicPDE} with the help of a 
	domain decomposition scheme. A similar setting can be found in 
	\cite{HansenHenningsson.2016} for $p = 2$. The case $p \geq 2$ for an 
	autonomous problem with a more restrictive domain decomposition around the 
	boundary can be found in \cite[Section~6]{EisenmannHansen.2018}.
	For $s\in \N$ let $\{ \Omega_{\ell} \}_{\ell =1}^{s}$ be a family of overlapping 
	subsets of $\Omega $ such that $\bigcup_{\ell =1}^s \Omega_{\ell} = \Omega$ 
	is fulfilled. Furthermore, let each $\Omega_{\ell}$, $\ell \in \{1,\dots,s\}$, be 
	either an open connected set with a Lipschitz boundary or a union of pairwise 
	disjoint open, connected sets $\Omega_{\ell,i}$ such that $\bigcup_{i=1}^{r} 
	\Omega_{\ell,i} = \Omega_{\ell}$ and $\Omega_{\ell,i}$ has a Lipschitz 
	boundary for every $\ell \in \{1,\dots,s\}$ and $i \in \{1,\dots,r\}$; see 
	Fig.~\ref{fig:dom}. 
 
  On these subdomains let the partition of unity $\{\chi_{\ell} \}_{\ell 
  =1}^{s}\subset W^{1,\infty}(\Omega)$ be given such that
	\begin{align*}
		\chi_{\ell} (x)>0\text{ for all }x\in\Omega_{\ell},
		\quad
		\chi_{\ell} (x) = 0\text{ for all }x\in\Omega\setminus\Omega_{\ell},
		\quad
		\sum_{\ell =1}^{s} \chi_{\ell}= 1
	\end{align*}
	for $\ell \in \{1,\dots,s\}$. For such a function $\chi_{\ell}$, $\ell \in 
	\{1,\dots,s\}$, the weighted Lebesgue space $L^p(\Omega_{\ell},\chi_{\ell})^d$ 
	consists of all measurable functions $v = (v_1,\dots,v_d) \colon \Omega_{\ell} 
	\to \R^d$ such that
	\begin{align*}
		\|(v_1,\ldots,v_{d})\|_{L^p(\Omega_{\ell},\chi_{\ell})^d}
		= \Big(\int_{\Omega_{\ell}}\chi_{\ell} |(v_1,\ldots,v_{d})|^p 
		\diff{x}\Big)^{\frac{1}{p}}
	\end{align*}
	is finite. The space $L^p(\Omega_{\ell},\chi_{\ell})^d$ is a reflexive Banach
	space; see \cite[Chapter~1]{DrabekEtAl.1997} and 
	\cite[Theorem~1.23]{AdamsFournier.2003}. Note that $L^p(\Omega_{\ell})^d$ 
	is a subspace of $L^p(\Omega_{\ell},\chi_{\ell})^d$ and it holds true that 	
	$\|v\|_{L^p(\Omega_{\ell},\chi_{\ell})^d} \leq \|v \|_{L^p(\Omega_{\ell})^d}$ for 
	every $v \in L^p(\Omega_{\ell})^d$.

	For $\left(H, \ska[H]{\cdot}{\cdot}, \|\cdot\|_H \right)$ we use $L^2(\Omega)$ 
	the space of square integrable functions on $\Omega$ with the usual norm and 
	inner product. The energetic spaces $V$ and $V_{\ell}$ are then given as
	\begin{align*}
		V &= \Bigl\{ u \in H:\text{ there exists a } v = (v_1,\dots,v_d) \in 
		L^p(\Omega)^d \text{ such 
		that }\\
		& \qquad \qquad  \int_{\Omega} u D_i \varphi \diff{x} = \int_{\Omega} 
		v_i 
		\varphi \diff{x} 
		\quad \text{for all } \varphi \in C^{\infty}_0(\Omega),\ i=1,\dots,d  \Bigr\}
    = W_0^{1,p}(\Omega)
	\end{align*}
	and
	\begin{align*}
		V_{\ell} &= \Bigl\{ u \in H:\text{ there exists a } v = (v_1,\dots,v_d) \in 
		L^p(\Omega_{\ell},\chi_{\ell})^d \text{ such that }\\
		& \qquad \qquad \int_{\Omega} u D_i (\chi_{\ell} \varphi) \diff{x} = 
		\int_{\Omega_{\ell}} v_i \chi_{\ell} \varphi \diff{x} \quad\text{for all } 
		\varphi 
		\in C^{\infty}_0(\Omega),\ i=1,\dots,d   \Bigr\}, 
	\end{align*}
	which are equipped with the norms 
	\begin{align*} 
		\|\cdot\|_{V}
		= \|\cdot\|_H + \|\nabla \cdot\|_{L^p(\Omega)^d}
		\quad\text{and}\quad
		\|\cdot\|_{V_{\ell}}
		= \|\cdot\|_H + \| \nabla \cdot \|_{L^p(\Omega_{\ell},\chi_{\ell})^d}.
	\end{align*}
	For $t\in [0,T]$ we introduce the operator $A(t) \colon V \to V^*$
	\begin{align*} 
		\dualV{A(t) u}{v} = \int_{\Omega} \alpha(t) |\nabla u|^{p-2} \nabla u \cdot 
		\nabla v \diff{x},
		\quad u,v\in V.
	\end{align*}
	Together with the partition of unity we define the decomposed energetic operators 
	$A_{\ell}(t) \colon V_{\ell} \to V^*_{\ell}$, $\ell\in \{ 1,\dots,s\}$,
	\begin{align*}
		\dualVell{A_{\ell}(t) u}{v}
		= \int_{\Omega_{\ell}} \chi_{\ell} \alpha(t) |\nabla u|^{p-2} \nabla u \cdot 
		\nabla v \diff{x},
		\quad u,v\in V_{\ell}, \quad t\in [0,T].
	\end{align*}
	It is also possible to allow for more general coefficients $\alpha \colon [0,T] 
	\times \overline{\Omega} \times \R^d \to \R^d$ where $\alpha(t,\cdot,\cdot)$ 
	fulfills the condition stated in \cite[Assumption~3]{EisenmannHansen.2018}.
	
	We assume that for $g \colon (0,T) \times\Omega \to \R$ the 
	abstract function $[f(t)](x) = g(t,x)$, $(t,x) \in (0,T) \times\Omega$, is an 
	element of $L^{q}(0,T;V^*)$. We exploit that $f(t) \in V^*$, $t \in (0,T)$, can be 
	represented by
	\begin{align*}
		\dualV{f(t)}{v} = \int_{\Omega} f^0(t) v \diff{x}
		+ \sum_{i = 1}^{d} \int_{\Omega} f^i(t) D_i v \diff{x}, \quad v \in V
	\end{align*}
	where $f^i(t) \in L^q(\Omega)$ for $i \in \{0,\dots,d\}$. These functions are not 
	necessarily unique unless we exchange $V = W^{1,p}(\Omega)$ by $V = 
	W_0^{1,p}(\Omega)$, compare \cite[Theorem~10.41, 
	Corollary~10.49]{Leoni.2009}. 
	This in mind, we introduce $f_{\ell}(t)$ for a.e. $t \in (0,T)$ as
	\begin{align*}
		\dualVell{f_{\ell}(t)}{v}
		= \int_{\Omega} \chi_{\ell} f^0(t) v \diff{x}
		+ \sum_{i = 1}^{d} \int_{\Omega} \chi_{\ell} f^i(t) D_i v \diff{x},\quad v \in 
		V_{\ell}.
	\end{align*}
	
	Note that in this type of setting, we can also consider homogeniuous Direchlet 
	boundary conditions in \eqref{2eq:parabolicPDE}. Then an additional condition on 
	the partition of unity becomes necessary. In this case, we have to make the 
	further assumption that for every function $\chi_{\ell}$ there exists 
	$\varepsilon_0>0$ such that for all $\varepsilon \in (0,\varepsilon_0)$
	\begin{align*}
		\Omega_{\ell}^{\varepsilon}
		= \{ x\in \Omega_{\ell} : \chi_{\ell}(x) \geq \varepsilon \}
	\end{align*}
	is a Lipschitz domain.
\end{example}

Further examples that fit our framework are a domain decomposition scheme for 
the porous medium equation as presented in 
\cite[Section~7]{EisenmannHansen.2018} or a source term splitting 
as in \cite[Section~3.3]{Eisenmann.2019}. An application to the third equation 
of \eqref{1eq:op} is presented in \cite{Temam.1968}. Numerical experiments for  
this equation and the $p$-Laplace equation can be found in 
\cite{EisenmannHansen.2018} and \cite{HansenOstermann.2008}, respectively.

\section{Solvability and a priori bounds for the discrete scheme}
\label{sec:solvability}

The abstract setting from the previous section in mind we are now well-prepared 
to state some properties of the solution of the numerical scheme 
\eqref{2eq:semiDiscret}. Since the scheme is implicit, we start to verify that 
\eqref{2eq:semiDiscret} is uniquely 
solvable. Once this is at hand, we can provide a priori bounds of the solution. 
These bounds are a crucial part of the further analysis and allow for the 
convergence analysis in Section~\ref{sec:convergence}

\begin{lemma}\label{lem:existence}
	Let Assumptions~\ref{ass:Aell} and \ref{ass:fell} be fulfilled.
	Then the semidiscrete problem \eqref{2eq:semiDiscret} is uniquely
	solvable.
\end{lemma}

\begin{proof}
	In order to prove the existence of the elements $(\U_{\ell}^i)_{i\in 
	\{0,\dots,N\}}$, $ \ell \in \{1,\dots,s\}$, that fulfill \eqref{2eq:semiDiscret}, 
	we argue inductively. Assuming that for $i \in \{1,\dots,N\}$ the 
	previous elements	$(\U_{\ell}^j)_{j\in \{0,\dots,i-1\}}$, $\ell \in 
	\{1,\dots,s\}$, exist in the corresponding spaces, we prove the existence of 
	$\U_{\ell}^i \in V_{\ell}$ for  every $\ell \in \{1,\dots,s\}$.
	The operator $I + s k \A_{\ell}^i$, $\ell \in \{1,\dots,s\}$, is strictly 
	monotone due to (3) of Assumption~\ref{ass:A}, i.e., it holds true that
	\begin{align*}
		\dualVell{(I + s k \A_{\ell}^i)v - (I + s k \A_{\ell}^i) w }{ v - w} > 0, 
		\quad v,w \in V_{\ell} \text{ with } v \neq w
	\end{align*}
	for every $\ell \in \{1,\dots,s\}$.
	Furthermore, $I + s k \A_{\ell}^i$ is radially continuous as $A_{\ell}(t)$, $\ell 
	\in \{1,\dots,s\}$, is radially continuous for every $t \in [0,T]$. It remains to 
	verify that the operator is coercive. Using (5) of Assumption~\ref{ass:A} and 
	the norm bound of Assumption~\ref{ass:spaces}, it follows 
	\begin{align*}
		\frac{\dualVell{\big(I + s k \A_{\ell}^i\big) v}{v}}{\|v\|_{V_{\ell}}}
		&\geq \frac{ \|v\|_H^2 + sk  \mu|v|_{V_{\ell}}^p
			- sk \lambda }{ c_{V_{\ell}} \big(\|v\|_{H} + |v|_{V_{\ell}}\big)  }   \\
		&\geq \frac{\min(1,sk \mu)}{c_{V_{\ell}}} \cdot \frac{\|v\|_H^2 + 
			|v|_{V_{\ell}}^p
		}{ \|v\|_{H} + |v|_{V_{\ell}} }
		- \frac{sk \lambda}{ c_{V_{\ell}} \big(\|v\|_{H} + |v|_{V_{\ell}} \big) }
		\to \infty
	\end{align*}
  as $\|v\|_{V_{\ell}} \to \infty$ for $v \in V_{\ell}$ and $\ell \in \{1,\dots,s\}$.
	Thus, for $\U^{i-1} = \frac{1}{s} \sum_{\ell = 1}^{s} \U_{\ell}^{i-1} \in H$,
	there exists a unique solution $\U_{\ell}^i \in
	V_{\ell}$ of
	\begin{align}\label{3eq:semiEq}
		\big( I + s k \A_{\ell}^i \big) \U_{\ell}^i = s k \F_{\ell}^i + \U^{i-1}
	\end{align}
	for every $\ell \in \{1,\dots,s\}$ due Browder--Minty theorem; see
	\cite[Theorem~2.14]{Roubicek.2013} for further details.
\end{proof}

We can now turn our attention to the a priori bounds. 

\begin{lemma}\label{lem:aprioriSum}
	Let Assumptions~\ref{ass:Aell} and \ref{ass:fell} be fulfilled.
	Then for the unique solution of \eqref{2eq:semiDiscret} there exist constants 
	$M, M' < \infty$ such that for every step size $k = \frac{T}{N}$ the a priori 
	bounds
	\begin{align}\label{3eq:apriori}
		\begin{split}
			\max_{n \in \{1,\dots,N\}} \Big(\frac{1}{s} \sum_{\ell = 1}^{s}
			\|\U_{\ell}^n\|^2_H \Big)
			+ \frac{1}{s} \sum_{i = 1}^{N} \sum_{\ell = 1}^{s} \|\U_{\ell}^i -
			\U^{i-1}\|^2_H
			+ k \sum_{i = 1}^{N} \sum_{\ell =1 }^{s} \| \U_{\ell}^i\|_{V_{\ell}}^p 
			\leq M
		\end{split}
	\end{align}
	and
	\begin{align}\label{3eq:aprioriDeriv}
		\frac{1}{k} \sum_{i=1}^{N} \Big\|\frac{\U^i - \U^{i-1}}{k} \Big\|^q_{V^*} 
    = k^{1-q} \sum_{i=1}^{N} \big\|\U^i - \U^{i-1} \big\|^q_{V^*} 
    \leq M'
	\end{align}
	are fulfilled.
\end{lemma}

\begin{proof}
	In the following, let $i \in  \{1,\dots,N\}$ and $\ell \in \{1,\dots,s\}$ be 
	arbitrary but fixed. Recall the identity
	\begin{align}\label{3eq:discreteDeriv}
		\ska[H]{\U_{\ell}^i - \U^{i-1} }{\U_{\ell}^i}
		= \frac{1}{2} \big( \|\U_{\ell}^i\|^2_H - \|\U^{i-1} \|^2_H
		+ \|\U_{\ell}^i - \U^{i-1} \|^2_H \big)
	\end{align}
	and the inequality $\|\cdot\|_{V_{\ell}} \leq c_1 \big( \|\cdot \|_H + 
	|\cdot |_{V_{\ell}} \big)$ with $c_1 = \max_{\ell \in \{1,\dots,s\}} 
	c_{V_{\ell}}$ stated in Assumption~\ref{ass:spaces}. Using the 
	weighted Young inequality, see \cite[Appendix~B.2.d]{Evans.1998}), we 
	obtain 
	\begin{align*}
		&\frac{1}{2k} \big( \|\U_{\ell}^i\|^2_H - \|\U^{i-1}\|^2_H
		+ \|\U_{\ell}^i - \U^{i-1} \|^2_H \big)
		+ \dualVell{s \A_{\ell}^i \U_{\ell}^i}{\U_{\ell}^i} \\
		&= \dualVell{s \F_{\ell}^i}{\U_{\ell}^i}
		\leq s c_1  \| \F_{\ell}^i\|_{V_{\ell}^*} \big( \|\U_{\ell}^i \|_H + 
		|\U_{\ell}^i |_{V_{\ell}} \big)\\
		&\leq s c_1  \| \F_{\ell}^i\|_{V_{\ell}^*} \|\U_{\ell}^i \|_H 
		+ s c_2  \| \F_{\ell}^i\|_{V_{\ell}^*}^q 
		+ \frac{s \mu}{2}  |\U_{\ell}^i |_{V_{\ell}}^p
	\end{align*}
	with $c_2 = c_1^q \frac{(p\mu)^{1-q}}{q 2^{1-q}}$. Thus, together 
	with the coercivity condition from Assumption~\ref{ass:A}~(5) it follows 
	that
	\begin{align}\label{3eq:proofapriori2}
    \begin{split}
      &\|\U_{\ell}^i\|^2_H - \|\U^{i-1}\|^2_H + \|\U_{\ell}^i - \U^{i-1} \|^2_H
      + k s \mu  |\U_{\ell}^i |_{V_{\ell}}^p \\
      &\leq 2k s c_1  \| \F_{\ell}^i\|_{V_{\ell}^*} \|\U_{\ell}^i \|_H 
      + 2 k s c_2  \| \F_{\ell}^i\|_{V_{\ell}^*}^q + 2 k s \lambda.
    \end{split}
	\end{align}
	Employing the specific structure of $\U^{i-1}$, we obtain
	\begin{align} \label{3eq:HoelderU}
    \begin{split}
      \| \U^{i-1}\|^2_H
      &= \Big\| \frac{1}{s} \sum_{\ell = 1}^{s} \U_{\ell}^{i-1}\Big\|^2_H
      \leq \frac{1}{s^2} \Big( \sum_{\ell = 1}^{s} \|\U_{\ell}^{i-1}\|_H
      \Big)^2\\
      &\leq \frac{1}{s^2} \sum_{\ell = 1}^{s} 1^2 \cdot
      \sum_{\ell = 1}^{s} \|\U_{\ell}^{i-1}\|_H^2
      = \frac{1}{s} \sum_{\ell = 1}^{s} \big\| \U_{\ell}^{i-1}\big\|^2_H
    \end{split}
	\end{align}
	for $i \in \{2,\dots,N\}$ due to the Cauchy--Schwarz inequality for sums.
	Inserting this inequality in \eqref{3eq:proofapriori2}, summing up from $\ell 
	= 1$ to $s$ as well as dividing by $s$, yields
	\begin{align*}
		&\frac{1}{s} \sum_{\ell = 1}^{s} \big( \|\U_{\ell}^i\|^2_H -
		\|\U_{\ell}^{i-1}\|^2_H + \|\U_{\ell}^i - \U^{i-1} \|^2_H \big) 
		+ k  \mu \sum_{\ell = 1}^{s}  |\U_{\ell}^i |_{V_{\ell}}^p \\
		&\leq 2k c_1 \sum_{\ell = 1}^{s} \|\F_{\ell}^i\|_{V_{\ell}^*} 
		\|\U_{\ell}^i \|_H 
		+ 2 k c_2 \sum_{\ell = 1}^{s}  \| \F_{\ell}^i\|_{V_{\ell}^*}^q 
		+ 2 k s \lambda
	\end{align*}
	for $i \in \{2,\dots,N\}$ and
	\begin{align*}
		&\frac{1}{s} \sum_{\ell = 1}^{s} \big( \|\U_{\ell}^1\|^2_H 
		+ \|\U_{\ell}^i - u_0 \|^2_H \big) 
		+ k  \mu \sum_{\ell = 1}^{s}  |\U_{\ell}^1 |_{V_{\ell}}^p \\
		&\leq \|u_0\|_H^2 + 2k c_1  \sum_{\ell = 1}^{s} \| 
		\F_{\ell}^1\|_{V_{\ell}^*} 
		\|\U_{\ell}^1 \|_H 
		+ 2 k c_2 \sum_{\ell = 1}^{s}  \| \F_{\ell}^1\|_{V_{\ell}^*}^q 
		+ 2 k s \lambda.
	\end{align*}
	After a summation from $i=1$ to $n \in \{1,\dots,N\}$ and using the 
	telescopic structure, we obtain
	\begin{align*}
		&\frac{1}{s} \sum_{\ell = 1}^{s} \|\U_{\ell}^n\|^2_H 
		+ \frac{1}{s} \sum_{i=1}^{n} \sum_{\ell = 1}^{s} \|\U_{\ell}^i - \U^{i-1} 
		\|^2_H
		+ k  \mu \sum_{i = 1}^{n} \sum_{\ell = 1}^{s}  |\U_{\ell}^i |_{V_{\ell}}^p \\
		&\leq \|u_0\|_H^2 + 2k c_1 \sum_{i = 1}^{n} \sum_{\ell = 1}^{s} \| 
		\F_{\ell}^i\|_{V_{\ell}^*} \|\U_{\ell}^i \|_H 
		+ 2 k c_2 \sum_{i = 1}^{n}\sum_{\ell = 1}^{s}  \| \F_{\ell}^i\|_{V_{\ell}^*}^q 
		+ 2 T s \lambda.
	\end{align*}
	For the right-hand side we can bound the summands using
	Assumption~\ref{ass:fell} and H\"older's inequality
	\begin{align}\label{3eq:boundF1}
    \begin{split}
      &k \sum_{i = 1}^{n} \sum_{\ell = 1}^{s} \| \F_{\ell}^i\|_{V_{\ell}^*}^{q}
      = k \sum_{i = 1}^{n} \sum_{\ell = 1}^{s} \Big\| \frac{1}{k}
      \int_{t_{i-1}}^{t_i} f_{\ell}(t) \diff{t} \Big\|_{V_{\ell}^*}^{q}\\
      &\leq  \sum_{i = 1}^{n} \sum_{\ell = 1}^{s}
      \int_{t_{i-1}}^{t_i} \| f_{\ell}(t) \|_{V_{\ell}^*}^{q} \diff{t}
      \leq s \| f \|_{L^{q}(0,T;V^*)}^{q}
    \end{split}
	\end{align}
	and
	\begin{align*}
		k \| \F_{\ell}^i\|_{V_{\ell}^*}
		\leq k \Big\| \frac{1}{k} \int_{t_{i-1}}^{t_i} f_{\ell}(t) \diff{t}
		\Big\|_{V_{\ell}^*}
		\leq \int_{t_{i-1}}^{t_i} \| f(t) \|_{V^*} \diff{t}.
	\end{align*}
	Thus, we get that
	\begin{align}\label{3eq:proofapriori3}
		\begin{split}
			&\frac{1}{s} \sum_{\ell = 1}^{s} \|\U_{\ell}^n\|^2_H 
			+ \frac{1}{s} \sum_{i = 1}^{n} \sum_{\ell = 1}^{s} \|\U_{\ell}^i - \U^{i-1} 
			\|^2_H
			+ k  \mu \sum_{i = 1}^{n} \sum_{\ell = 1}^{s}  |\U_{\ell}^i |_{V_{\ell}}^p \\
			&\leq \|u_0\|_H^2 + 2k c_1   \sum_{i = 1}^{n} 
			\int_{t_{i-1}}^{t_i} \| f(t) \|_{V^*} \diff{t} \sum_{\ell = 1}^{s} 
			\|\U_{\ell}^i 
			\|_H 
			+ 2 k s c_2 \| f \|_{L^{q}(0,T;V^*)}^{q}
			+ 2 T s \lambda. 
		\end{split}
	\end{align}
	As this is fulfilled for every $n \in \{1,\dots,N\}$, it also follows that
	\begin{align*}
		&\max_{n \in \{1,\dots,N\}} \Big(\frac{1}{s} \sum_{\ell = 1}^{s} 
		\|\U_{\ell}^n\|^2_H 
		+ \frac{1}{s} \sum_{i=1}^{n} \sum_{\ell = 1}^{s} \|\U_{\ell}^i - \U^{i-1} 
		\|^2_H
		+ k  \mu \sum_{i = 1}^{n} \sum_{\ell = 1}^{s}  |\U_{\ell}^i 
		|_{V_{\ell}}^p\Big) \\
		&\leq \|u_0\|_H^2 + 2k s c_1 \| f \|_{L^1(0,T;V^*)} 
		\max_{n \in \{1,\dots,N\}} \Big(\frac{1}{s}\sum_{\ell = 1}^{s} 
		\|\U_{\ell}^n\|_H^2\Big)^{\frac{1}{2}} \\
		 &\qquad + 2 k s c_2 \| f \|_{L^{q}(0,T;V^*)}^q + 2 T s \lambda. 
	\end{align*}
	We abbreviate the terms 
	\begin{align*}
		x^2 &= \max_{n \in \{1,\dots,N\}} \Big(\frac{1}{s} \sum_{\ell = 1}^{s} 
		\|\U_{\ell}^n\|^2_H 
		+ \frac{1}{s} \sum_{i=1}^{n} \sum_{\ell = 1}^{s} \|\U_{\ell}^i - \U^{i-1} 
		\|^2_H + k  \mu \sum_{i = 1}^{n} \sum_{\ell = 1}^{s}  |\U_{\ell}^i 
		|_{V_{\ell}}^p\Big)\\
		a &= k s c_1 \| f \|_{L^1(0,T;V^*)}\\
		b^2 &= \|u_0\|_H^2 + 2 k s c_2 \| f \|_{L^{q}(0,T;V^*)}^{q}
		+ 2 T s \lambda
	\end{align*}
	to obtain $x^2 \leq 2ax + b^2$. This implies, in particular, that
  \begin{align*}
    (x-a)^2 = x^2 -2ax +a^2 \leq a^2 +b^2.
  \end{align*}
  Taking the square root on both sides, this yields
  \begin{align*}
    |x-a| \leq \sqrt{ a^2 +b^2} \leq a + b.
  \end{align*}
  As $x-a \leq |x-a|$ is fulfilled, we obtain $x \leq 2a + b$ after adding $a$ to 
  both sides of the inequality. This shows that
	\begin{align*}
		\max_{n \in \{1,\dots,N\}} \Big(\frac{1}{s} \sum_{\ell = 1}^{s} 
		\|\U_{\ell}^n\|^2_H 
		+ \frac{1}{s} \sum_{i=1}^{n} \sum_{\ell = 1}^{s} \|\U_{\ell}^i - \U^{i-1} 
		\|^2_H + k  \mu \sum_{i = 1}^{n} \sum_{\ell = 1}^{s}  |\U_{\ell}^i 
		|_{V_{\ell}}^p\Big) \leq M_1,
	\end{align*} 
	where $M_1 \geq 0$ is independent of $k$.  Using the norm inequality from 
	Assumption~\ref{ass:spaces}, this implies that there exists $M_2 \geq 0$, 
	which does not depend on $k$, such that
	\begin{align*}
		\Big(k \sum_{i = 1}^{N} \sum_{\ell = 1}^{s}  \|\U_{\ell}^i 
		\|_{V_{\ell}}^p\Big)^{\frac{1}{p}}
		\leq c_1 \Big(k \sum_{i = 1}^{N} \sum_{\ell = 1}^{s}  \|\U_{\ell}^i 
		\|_H^p\Big)^{\frac{1}{p}}
		+ c_1 \Big(k \sum_{i = 1}^{N} \sum_{\ell = 1}^{s}  |\U_{\ell}^i 
		|_{V_{\ell}}^p\Big)^{\frac{1}{p}}
		\leq M_2.
	\end{align*}
	Altogether, we have proved the first a priori bound \eqref{3eq:apriori}.

	In order to prove \eqref{3eq:aprioriDeriv}, we test \eqref{2eq:semiDiscret} 
	with $v \in V$ and use Assumption~\ref{ass:A}~(4) to see that
	\begin{align*}
		\skaB[H]{\frac{\U^i - \U^{i-1}}{k}}{v}
		&= \frac{1}{s}\sum_{\ell = 1}^{s} \skaB[H]{\frac{\U_{\ell}^i - \U^{i-1}}{k}}{v} 
		\\
		&= \sum_{\ell = 1}^{s} \big( \dualVell{\F_{\ell}^i}{v} - \dualVell{\A_{\ell}^i
			\U_{\ell}^i }{v} \big)\\
		&\leq c_3 \|v\|_{V} \Big(\sum_{\ell = 1}^{s} \| \F_{\ell}^i\|_{V_{\ell}^*}
		+ \beta \sum_{\ell = 1}^{s} \big(1 + \|\U_{\ell}^i \|_{V_{\ell}}^{p-1} 
		\big)\Big)
	\end{align*}
	for $i \in \{1,\dots,N\}$, where $c_3$ is the maximal embedding constant of 
	$V$ into $V_{\ell}$ for $\ell \in \{1,\dots,s\}$.
	Thus, we can estimate the $V^*$-norm by
	\begin{align*}
		k^{-1} \big\|\U^i - \U^{i-1} \big\|_{V^*}
		\leq  c_3\Big(\sum_{\ell = 1}^{s} \| \F_{\ell}^i\|_{V_{\ell}^*}
		+ \beta \sum_{\ell = 1}^{s} \big(1 + \|\U_{\ell}^i \|_{V_{\ell}}^{p-1} 
		\big)\Big).
	\end{align*}
	This bound can be used to see that there exists $M' \geq 0$ such that
	\begin{align*}
		&\Big(k^{1-q} \sum_{i=1}^{N} \big\|\U^i - \U^{i-1} 
		\big\|^q_{V^*}\Big)^{\frac{1}{q}}\\
		&\leq c_3\Big(k \sum_{i=1}^{N} \Big(\sum_{\ell = 1}^{s} \|   
		\F_{\ell}^i\|_{V_{\ell}^*} + \beta \sum_{\ell = 1}^{s} \big(1 + \|\U_{\ell}^i 
		\|_{V_{\ell}}^{p-1} \big) \Big)^q\Big)^{\frac{1}{q}}\\
		&\leq  c_3\Big(k \sum_{i=1}^{N} \sum_{\ell = 1}^{s} \| 
		\F_{\ell}^i\|_{V_{\ell}^*}^q \Big)^{\frac{1}{q}}
		+  c_3 \beta (Ts)^{\frac{1}{q}}
		+  c_3 \beta \Big(k \sum_{i=1}^{N} \sum_{\ell = 1}^{s} \|\U_{\ell}^i 
		\|_{V_{\ell}}^p \Big)^{\frac{1}{q}}
		\leq M'.
	\end{align*}
	Due to the first a priori bound \eqref{3eq:apriori} and \eqref{3eq:boundF1} 
	the constant $M'$ is independent of $k$.
\end{proof}

\section{Convergence analysis} \label{sec:convergence}

In the following, we introduce prolongations of the solution of the discrete 
problem \eqref{2eq:semiDiscret} to the interval $[0,T]$. The main goal of this 
section is to prove that the sequence of such prolongations converges to the 
exact solution $u$ of \eqref{1eq:PDE}. 
Corresponding to the grid $0 = t_0 < t_1 < \dots <t_N = T$ with $k = 
\frac{T}{N}$ and $t_n = n k$, $n \in \{0,\dots,N\}$, we construct piecewise 
constant and piecewise linear functions on the interval $[0,T]$.
We consider the piecewise constant functions for $t \in 
(t_{n-1},t_n]$, $n \in \{1,\dots,N\}$, and $\ell \in \{1,\dots,s\}$ given by
\begin{align} \label{4eq:defConstInt}
	U^k_{\ell}(t) = \U_{\ell}^n, \quad
	U^k(t) = \U^n, \quad
	A^k_{\ell}(t) = \A_{\ell}^n, \quad \text{and} \quad
	f^k_{\ell}(t) = \F_{\ell}^n
\end{align}
as well as the piecewise linear function
\begin{align} \label{4eq:defLinInt}
	\tilde{U}^k(t) = \U^{n-1} + \frac{t - t_{n-1}}{k}
	(\U^{n} - \U^{n-1})
\end{align}
with $U^k_{\ell}(0) = U^k(0) = \tilde{U}^k(0) = u_0$, $ A^k_{\ell}(0) = 
\A_{\ell}^1$ and $f^k_{\ell}(0) = \F_{\ell}^1$.
As we consider step sizes $k = \frac{T}{N}$ for $N \in \N$, we denote
the sequences $\big(U^{\frac{T}{N}}_{\ell} \big)_{N \in \N}$ as
$(U^k_{\ell} )_{k >0}$ for $\ell \in \{1,\dots,s\}$ in the following to keep the
notation more compact.
The same simplification in notation is used for the other functions
introduced above. Due to the a priori bound \eqref{3eq:apriori} we see that
\begin{align*}
	U^k_{\ell} \in L^p(0,T;V_{\ell}) \cap L^{\infty}(0,T;H), \
	U^k, \tilde{U}^k \in L^{\infty}(0,T;H), \ \text{and}
	\
	f^k_{\ell} \in L^q(0,T;V_{\ell}^*).
\end{align*}
Furthermore, due to Lemma~\ref{lemma:ABochner} the operator $A_{\ell}^k$
maps the space $L^p(0,T;V_{\ell})$ into $L^q(0,T;V_{\ell}^*)$.
Using the prolongations introduced above, we can state a discrete version of 
the differential equation. We first note that after summing up 
\eqref{2eq:semiDiscret} from $1$ to $s$ and dividing by $s$, we obtain
\begin{align*}
	\frac{1}{ks}\sum_{\ell = 1}^{s} \big( \U_{\ell}^n - \U^{n-1} \big)
	+ \sum_{\ell = 1}^{s} \A_{\ell}^n \U_{\ell}^n
	= \sum_{\ell = 1}^{s} \F_{\ell}^n \quad \text{in }V^*.
\end{align*}
Thus, we see that
\begin{align}
	\begin{split} \label{4eq:pdeSemi}
		\begin{cases}
			(\tilde{U}^k)'(t)
			+ \sum_{\ell = 1}^{s} A^k_{\ell}(t) U_{\ell}^k(t)
			= \sum_{\ell = 1}^{s} f^k_{\ell}(t)
			\quad &\text{in }V^*,
			\quad t \in (0, T),\\
			U^k(0) = \tilde{U}^k(0) = u_0 \quad &\text{in }H,
		\end{cases}
	\end{split}
\end{align}
where $(\tilde{U}^k)'$ is the weak derivative of $\tilde{U}^k$. In the 
following, we will consider the limiting process of all the
appearing terms to connect to the original problem \eqref{2eq:pde} with 
\eqref{4eq:pdeSemi}.

\begin{lemma} \label{lem:ConvergenceA}
	Let Assumption~\ref{ass:Aell} be fulfilled and let $W\in L^p(0, T;V)$ be
	given.
	For $\ell \in \{1,\dots,s\}$ it follows that $A_{\ell}^k(t) W(t) \to A_{\ell}(t) 
	W(t)$ in $V_{\ell}^*$ as $k \to 0$ for a.e. $t\in (0,T)$. Furthermore, it holds 
	true that $A_{\ell}^k W \to A_{\ell} W$ in $L^q(0,T;V_{\ell}^*)$ as $k \to 0$.
\end{lemma}

\begin{proof}
	Let $\ell \in \{1,\dots,s\}$ and $\varepsilon> 0$ be arbitrary. Due to the 
	continuity condition on $A_{\ell}$, for almost every $t \in (0,T)$ we find 
	$\delta >0$ such that for all $k < \delta$ it follows that
	\begin{align*}
		\| A_{\ell}^k(t) W(t) - A_{\ell}(t) W(t) \|_{V^*}
		= \| A_{\ell}(t_n) W(t) - A_{\ell}(t) W(t) \|_{V^*}
		\leq \varepsilon,
	\end{align*}
	where $t$ is within an interval $(t_{n-1},t_n]$, $n \in \{1,\dots,N\}$.
	The second assertion of the lemma is a consequence of Lebesgue's 
	theorem of dominated convergence and the boundedness condition (4) 
	from Assumption~\ref{ass:A}.
\end{proof}

\begin{lemma} \label{lem:ConvergenceF}
	Let Assumption~\ref{ass:fell} be fulfilled. Then it follows that $f_{\ell}^k \to
	f_{\ell}$ in $L^{q}(0,T;V_{\ell}^*)$, $\ell \in \{1,\dots,s\}$, as $k \to 0$.
\end{lemma}

\begin{proof} 
	The statement above can easily be verified for a function from the space 
	$C([0,T] ;V_{\ell}^*)$. As the space $C([0,T] ;V_{\ell}^*)$ is a dense subspace 
	of $L^{q}(0,T;V_{\ell}^*)$ for $\ell \in \{1,\dots,s\}$ a density argument can 
	be used to verify the claimed statement.
\end{proof}

\begin{lemma}\label{lem:ConvSubSequenceS}
	Let Assumptions~\ref{ass:Aell} and \ref{ass:fell} be fulfilled.
	Then there exists a subsequence  $(k_{i})_{i \in \N}$ of step sizes
	$k_i = \frac{T}{N_i}$ and $U \in \W$ such that
	\begin{align*}
		U_{\ell}^{k_i} \weak U \quad \text{in } L^{p}(0, T; V_{\ell})
		\quad \text{and} \quad
		U_{\ell}^{k_i} \weaks U, \quad U^{k_i} \weaks U \quad \text{in } 
		L^{\infty}(0,T; H),
	\end{align*}
	as well as
	\begin{align*}
		\tilde{U}^{k_i} \weaks U \quad \text{in } L^{\infty}(0, T; H)
		\quad \text{and} \quad
		(\tilde{U}^{k_i})' \weak U' \quad \text{in } L^q(0, T;V^*)
	\end{align*}
	for every $\ell \in \{ 1,\dots, s \}$ as $i \to \infty$. Here, $U'$ denotes 
	the weak derivative of $U$.
\end{lemma}

\begin{proof}
	In the following proof, we do not distinguish between subsequences by 
	notation. Using Lemma~\ref{lem:aprioriSum}, we obtain that
	\begin{align*}
		&\| U_{\ell}^{k}\|_{L^{\infty}(0,T;H)}^2
		\leq s M,\\
		&\| U^{k}\|_{L^{\infty}(0,T;H)}^2
		= \| \tilde{U}^{k}\|_{L^{\infty}(0,T;H)}^2
		\leq M,\\
		&\| U_{\ell}^{k}\|_{L^p(0,T;V_{\ell})}^p = k \sum_{i = 1}^{N} \|
		\U_{\ell}^i\|_{V_{\ell}}^p
		\leq  M,\\
		&\big\| (\tilde{U}^{k})' \big\|_{L^q(0,T;V^*)}^q = k^{1-q} 
		\sum_{i=1}^{N}
		\big\|\U^n - \U^{n-1} \big\|^q_{V^*}
		\leq M'.
	\end{align*}
	Therefore, the sequence $(U_{\ell}^{k})_{k >0}$ is bounded in
	$L^p(0, T;V_{\ell})$ as well as $L^{\infty}(0, T;H)$,
	$(\tilde{U}^{k})_{k >0}$ is bounded in $L^{\infty}(0, T;H)$, and
	$\big( (\tilde{U}^{k})' \big)_{k >0}$ is bounded in $L^q(0, T;V^*)$.
	Since $L^p(0, T;V_{\ell})$ is a reflexive Banach space and
	$L^{\infty}(0, T;H)$ is the dual space of the separable Banach space
	$L^1(0, T;H)$, there exists a subsequence of $(U_{\ell}^{k})_{k>0}$ and
	$U_{\ell} \in  L^p(0, T;V_{\ell}) \cap L^{\infty}(0, T;H)$ such that
	\begin{align*}
		U_{\ell}^{k} \weak U_{\ell} \quad \text{in } L^p(0,T;V_{\ell})
		\quad \text{and} \quad
		U_{\ell}^{k} \weaks U_{\ell} \quad \text{in } L^{\infty}(0, T;H)
	\end{align*}
	as $ k \to 0$. Analogously, there exist a suitable further subsequence, 
	$\tilde{U} \in L^{\infty}(0, T;H)$ and $W \in L^q(0, T; V^*)$ such that
	\begin{align*}
		\tilde{U}^{k} \weaks \tilde{U} \quad \text{in } L^{\infty}(0, T;H)
		\quad \text{and} \quad
		(\tilde{U}^{k})' \weak W \quad \text{in } L^q(0, T;V^*)
	\end{align*}
	as $k \to 0 $.
	In the following, we prove that $U_{1} = \dots = U_{s} =: U$ is
	fulfilled. As $\bigcap_{\ell = 1}^{s} V_{\ell}  = V$ and the norm $\sum_{\ell = 
	1}^{s}\|\cdot\|_{V_{\ell}}$ is equivalent to $\|\cdot\|_V$ this implies that $U 
	\in L^p(0,T;V)$.
	We show that $U_{1} = U_{2}$ in $L^p(0,T;V_1\cap V_2) \cap 
	L^{\infty}(0,T;H)$, the 
	other equalities follow in an analogous manner. 
	We can write  
	\begin{align*}
		U_{1}^{k}(t) - U_{2}^{k}(t)
		&= \U_{1}^n - \U^{n-1} - (\U_{2}^n - \U^{n-1})\\
		&= k s \big( \F_{1}^n - \A_{1}^n \U_{1}^n \big)
		- k s \big( \F_{2}^n - \A_{2}^n \U_{2}^n \big)\\
		&= s \int_{t_{n-1}}^{t_n} \big(\big( f_{1}(\tau) - A_{1}(\tau)
		U_{1}^k(\tau) \big) - \big(  f_{2}(\tau) - A_{2}(\tau)
		U_{2}^k(\tau)  \big)\big) \diff{\tau}
	\end{align*}
	for $t\in (t_{n-1},t_n]$, $n\in \{1,\dots,N\}$,  as $U_1^{k}(t_{n-1}) = \U^{n-1} 
	= U_2^{k}(t_{n-1})$ holds true by the construction of our scheme.
	Therefore, we obtain
	\begin{align*}
		\| U_1^{k}(t) - U_2^{k}(t) \|_{V^*}
		&\leq s \sum_{\ell  = 1}^2  \int_{t_{n-1}}^{t_n} \| f_{\ell}(\tau) - 
		A_{\ell}(\tau)
		U_{\ell}^{k}(\tau)\|_{V^*} \diff{\tau}\\
		&\leq s k^{\frac{1}{p}} \sum_{\ell  = 1}^2 \Big( \int_{t_{n-1}}^{t_n}
		\|f_{\ell}(\tau) - A_{\ell}(\tau) U_{\ell}^{k}(\tau)\|_{V^*}^{q}
		\diff{\tau}\Big)^{\frac{1}{q}},
	\end{align*}
	where we used H\"older's inequality in the last step. We can bound the
	integrals by
	\begin{align*}
		&\Big(\int_{t_{n-1}}^{t_n} \|f_{\ell}(\tau) - A_{\ell}(\tau)
		U_{\ell}^{k}(\tau)\|_{V^*}^{q} \diff{\tau}\Big)^{\frac{1}{q}}\\
		&\leq \Big(\int_{t_{n-1}}^{t_n} \|f_{\ell}(\tau)\|_{V^*}^{q} 
		\diff{\tau}\Big)^{\frac{1}{q}}
		+ \Big(\int_{t_{n-1}}^{t_n} \|A_{\ell}(t_n)
		\U_{\ell}^{n}\|_{V^*}^{q} \diff{\tau}\Big)^{\frac{1}{q}}\\
		&\leq \|f_{\ell}\|_{L^q(0,T;V_{\ell}^*)}
		+ k^{\frac{1}{q}} \beta \big(1 + \|\U_{\ell}^{n}\|_{V_{\ell}}^{p-1} \big),
	\end{align*}
	for $\ell \in \{1,2\}$ which is bounded independently of $n$ and $k$ due to 
	the a priori bound \eqref{3eq:apriori}.
	Thus, it follows that $\| U_1^{k}(t) - U_2^{k}(t) \|_{V^*} \to 0$ as $k \to 0$ 
	for every $t \in [0,T]$ and it is bounded by a  constant independent of $t$. Using
	Lebesgue's theorem of dominated convergence and the fact that $L^q(0,T;V^*) \incl 
	L^1(0,T;V^*)$, it follows that
	\begin{align*}
		\| U_{1}^k - U_{2}^k \|_{L^1(0,T;V^*)}
		\leq c_1 \| U_{1}^k - U_{2}^k \|_{L^q(0,T;V^*)} \to 0 \quad \text{as } k \to 0  
	\end{align*}
	for $c_1 \in (0,\infty)$.
	This shows that $U_1 - U_2 = 0$ in $L^1(0,T;V^*)$. By assumption the 
	embedding $L^p(0,T;V_1\cap V_2) \cap L^{\infty}(0,T;H) \incl  
	L^1(0,T;V^*)$ is continuous.
	The injectivity of the embedding operator implies that both $U_1 - U_2 = 
	0$ in $L^p(0,T;V_1 \cap V_2)$ and $U_1 - U_2 = 0$ in $L^{\infty}(0,T;H)$.
	The limits $U$ and $\tilde{U}$ coincide in $L^1(0,T;V^*)$ since
	\begin{align*}
		\| U_1^k - U_2^k \|_{L^1(0,T;V^*)}^q
		&\leq c_1^q \int_{0}^{T} \| U^k(t) - \tilde{U}^k(t) \|_{V^*}^q \diff{t}\\
		&= \sum_{n=1}^{N} \int_{t_{n-1}}^{t_n} \Big\| \U^{n} - \U^{n-1} -
		\frac{t - t_{n-1}}{k} (\U^{n} - \U^{n-1}) \Big\|_{V^*}^q \diff{t}\\
		&=  \frac{1}{k^q} \sum_{n=1}^{N} \| \U^{n} - \U^{n-1}\|_{V^*}^q
		\int_{t_{n-1}}^{t_n} (t_n - t)^q  \diff{t}\\
		&=  \frac{k}{q} \sum_{n=1}^{N} \| \U^{n} - \U^{n-1}\|_{V^*}^q
		\leq  \frac{k^q}{q} M' \to 0
		\quad \text{ as } k \to 0,
	\end{align*}
	where we used the a priori bound \eqref{3eq:aprioriDeriv}. Making use of 
	the continuous embedding $L^{\infty}(0,T;H) \incl L^1(0,T;V^*)$, it follows that 
	$U$ and $\tilde{U}$ coincide in $L^{\infty}(0,T;H)$.

	Last, we prove that the limit $W \in L^q(0,T;V^*)$ is the weak
	derivative of $U$. To this end, let $v \in V$ and $\varphi \in
	C_c^{\infty}(0,T)$ be arbitrary. Using $\tilde{U}^k \weaks U$ in 
	$L^{\infty}(0,T;H)$ as $k \to 0$, it yields
	\begin{align*}
		-\int_{0}^{T} \dualV{W(t) }{v} \varphi(t) \diff{t}
		&= -\lim_{k \to 0} \int_{0}^{T}
		\dualV{(\tilde{U}^k)'(t) }{v} \varphi(t) 
		\diff{t}\\
		&= \lim_{k \to 0} \int_{0}^{T} \ska[H]{\tilde{U}^k(t) }{v} \varphi'(t) 
		\diff{t}
		= \int_{0}^{T} \ska[H]{U(t) }{v} \varphi'(t) \diff{t}.
	\end{align*}
	Applying \cite[Kapitel~IV, Lemma~1.7]{GGZ.1974}, we obtain $W = U'$ in 
	$L^q(0,T;V^*)$.
\end{proof}

The next lemma is an auxiliary result to identify the limit from the previous
lemma with the solution of \eqref{2eq:pde}.

\begin{lemma} \label{lem:Minity}
	Let Assumption~\ref{ass:Aell} be fulfilled and let the operator $A_{\ell}^k 
	\colon L^p(0,T;V_{\ell}) \to L^q(0,T;V_{\ell}^*)$ be given as in
	\eqref{4eq:defConstInt}. Then for a sequence $\big( W_{\ell}^k
	\big)_{k >0}$ in $L^{p}(0,T ;V_{\ell})$, $\ell \in \{1,\dots,s\}$, and an
	element $W \in L^{p}(0,T ;V)$ such that
	\begin{align*}
		W_{\ell}^k \weak W
		\quad\text{in } L^{p}(0,T ;V_{\ell})
		\quad \text{and}\quad
		A_{\ell}^k W_{\ell}^k \weak B_{\ell}
		\quad\text{in } L^{q}(0,T ;V_{\ell}^*)
	\end{align*}
	for every $\ell \in \{1,\dots,s\}$ as $k \to 0$ and $\sum_{\ell = 1}^{s}
	B_{\ell} = B \in L^q(0,T;V^*)$ with
	\begin{align}\label{4eq:mintyAss}
		\limsup_{k \to 0} \sum_{\ell = 1}^{s} \int_{0}^{T}
		\dualVell{A_{\ell}^k(t) W_{\ell}^k(t)}{W_{\ell}^k(t)} \diff{t}
		\leq \int_{0}^{T} \dualV{B(t) }{W(t)} \diff{t}
	\end{align}
	it follows that $ AW = \sum_{\ell = 1}^{s} A_{\ell} W= \sum_{\ell = 1}^{s} 
	B_{\ell} = B$ in $L^q(0,T;V^*)$.
\end{lemma}

\begin{proof}
	Due to the monotonicity of $A_{\ell}(t)$, $t \in [0,T]$ and $\ell \in 
	\{1,\dots,s\}$,
	we can write for every $X \in L^p(0,T;V)$
	\begin{align*}
		\sum_{\ell = 1}^{s} \int_{0}^{T}
		\dualVell{A_{\ell}^k(t) W_{\ell}^k(t)
			- A_{\ell}^k(t) X(t) } {W_{\ell}^k(t) - X(t) } \diff{t} 
		\geq 0.
	\end{align*}
	Thus, using Lemma~\ref{lem:ConvergenceA} it follows that
	\begin{align*}
		&\sum_{\ell = 1}^{s} \int_{0}^{T} \dualVell{
			A_{\ell}^k(t) W_{\ell}^k(t)}{W_{\ell}^k(t)} \diff{t}\\
		&\geq \sum_{\ell = 1}^{s} \int_{0}^{T}
		\big( \dualVell{ A_{\ell}^k(t) W_{\ell}^k(t)}{ X(t) } +
		\dualVell{ A_{\ell}^k(t) X(t) }{W_{\ell}^k(t) - X(t)} \big) \diff{t}\\
		&\stackrel{k \to 0}{\longrightarrow}
		\sum_{\ell = 1}^{s} \int_{0}^{T}
		\big( \dualVell{B_{\ell}(t)}{ X(t) } + \dualVell{A_{\ell}(t) X(t) }{ W(t) - 
		X(t) }
		\big) \diff{t}\\
		&= \int_{0}^{T} \big( \dualV{B(t)}{ X(t) }
		+\dualV{A(t) X(t) }{ W(t) - X(t) } \big) \diff{t}.
	\end{align*}
	This implies
	\begin{align*}
		&\liminf_{k \to 0} \sum_{\ell = 1}^{s} \int_{0}^{T}
		\dualVell{ A_{\ell}^k(t)
			W_{\ell}^k(t)}{W_{\ell}^k(t)} \diff{t}\\
		&\geq \int_{0}^{T} \big(\dualV{B(t)}{ X(t) }
		+ \dualV{ A(t) X(t) }{ W(t) - X(t)} \big) \diff{t}.
	\end{align*}
	Applying \eqref{4eq:mintyAss}, this yields
	\begin{align*}
		&\int_{0}^{T} \dualV{B(t) }{W(t)} \diff{t} \\
		&\geq \int_{0}^{T} \big( \dualV{B(t)}{ X(t) }
		+ \dualV{A(t) X(t) }{W(t) - X(t) } \big) \diff{t}.
	\end{align*}
	The assertion of the lemma follows by the Minty monotonicity trick, 
	where $X =  W \pm \theta \tilde{X}$ for $\theta \in [0,1]$ and 
	$\tilde{X} \in L^p(0,T; V)$ is inserted in the inequality. 
	Dividing by $\theta$ and considering $\theta \to 0$ then yields $A W= B$ 
	in $L^q(0,T;V^*)$. See, e.g., \cite[Lemma~2.13]{Roubicek.2013} for further 
	details.
\end{proof}

Combining the prior lemmas, we can now state one of the main results of
this section. We prove that the limit of the sequence of prolongations 
is the solution of \eqref{2eq:pde}.

\begin{theorem}\label{thm:limitWeak}
	Let Assumptions~\ref{ass:Aell} and \ref{ass:fell} be fulfilled. Let $u$ be
	the solution of \eqref{2eq:pde}. Then for step sizes $k = \frac{T}{N}$ the
	sequences $(U_{\ell}^k)_{k>0}$, $(U^k)_{k>0}$ and  $(\tilde{U}^k)_{k>0}$ 
	defined in \eqref{4eq:defConstInt} and \eqref{4eq:defLinInt}, respectively, 
	fulfill
	\begin{alignat}{2}
		\label{4eq:convLp}
		U_{\ell}^k &\weak u \quad &&\text{in }
		L^p(0,T;V_{\ell}), \\
		\label{4eq:convH}
		U^k(t)&\weak u(t) \quad
		&&\text{in } H, \\
		\label{4eq:convLinf}
		U^k \weaks u ,\quad
		\tilde{U}^k &\weaks u \quad
		&&\text{in } L^{\infty}(0,T;H), \\
		\label{4eq:convDeriv}
		(\tilde{U}^k)' &\weak u' \quad
		&&\text{in } L^q(0,T;V^*),\\
		\label{4eq:convAu}
		\sum_{\ell = 1}^{s} A_{\ell}^k U_{\ell}^k &\weak A u \quad
		&&\text{in } L^q(0,T;V^*)
	\end{alignat}
	as $k \to 0$ for $t\in [0,T]$ and $\ell \in \{1,\dots,s\}$.
\end{theorem}

\begin{proof}
	In the following, we will not distinguish between different subsequences by 
	notation. 
	Due to Lemma~\ref{lemma:ABochner} as well as the a priori bound
	\eqref{3eq:apriori} there exists a constant $\tilde{M} > 0$ such that for every 
	$k>0$, $t \in [0,T]$, and $\ell \in \{1,\dots,s\}$
	\begin{align*}
		\| A_{\ell}^k U_{\ell}^k \|_{L^{q}(0, T;V_{\ell}^*)} \leq \tilde{M}
		\quad \text{and} \quad
		\| U^k(t) \|_H
		\leq \tilde{M}
	\end{align*}
	is fulfilled. Therefore, we can extract a subsequence of step sizes such 
	that there exits $B_{\ell} \in L^{q}(0, T;V_{\ell}^*)$ and $y_t \in H$ with
	\begin{align}\label{4eq:LimitProof2}
		A_{\ell}^k U_{\ell}^k \weak B_{\ell} \quad \text{in }
		L^{q}(0, T;V_{\ell}^*)
		\quad \text{and} \quad
		U^k(t) \weak y_t \quad
		\text{in } H
	\end{align}
	as $k \to 0$ for $\ell \in \{1,\dots,s\}$. In the following, we abbreviate $B := 
	\sum_{\ell  = 1}^{s} B_{\ell}$.
	Next we identify the derivative of $U$ with equation \eqref{2eq:pde}. Due to 
	Lemma~\ref{lem:ConvergenceF} and Lemma~\ref{lem:ConvSubSequenceS}, 
	the following equality holds true
	\begin{align*}
		U'
		= \wlim_{k \to 0} (\tilde{U}^k)'
		= \wlim_{k \to 0} \sum_{\ell = 1}^{s} \big(f_{\ell}^k - A_{\ell}^k
		U_{\ell}^k\big)
		= f - B \quad \text{in } L^q(0,T;V^*),
	\end{align*}
	where $\wlim$ denotes the weak limit. The limit $U$ obtained in 
	Lemma~\ref{lem:ConvSubSequenceS} is an element of $\W \incl C([0, T];H)$. 
	Thus, we can work with the continuous representative of $U$ in the following.
	This in mind, we prove $y_{t} = U(t)$ and $u_0 = U(0)$ for $t \in [0,T]$. To
	this end, let $ v\in V$ and $\varphi  \in C([0,T]) \cap C^1(0,T)$ be arbitrary.
	Recalling the equation for the time discrete values \eqref{4eq:pdeSemi}, we 
	can write
	\begin{align*}
		& \ska[H]{U(t)}{v} \varphi(t) - \ska[H]{U(0)}{v} \varphi(0)
		- \int_{0}^{t} \ska[H]{U(\tau)}{v} \varphi'(\tau) \diff{\tau}\\
		& = \int_{0}^{t} \dualV{U'(\tau)}{v} \varphi(\tau) \diff{\tau}\\
		& = \sum_{\ell = 1}^{s}\int_{0}^{t} \dualVell{f_{\ell}(\tau) -
			B_{\ell}(\tau)}{v} \varphi(\tau) \diff{\tau}\\
		& = \lim_{k \to 0}
		\Big(\int_{0}^{t_n} \dualV{ (\tilde{U}^k)' (\tau)
			+ \sum_{\ell = 1}^{s} \big( A_{\ell}^k(\tau) U_{\ell}^k(\tau)
			- B_{\ell} (\tau)\big) }  {v} \varphi(\tau) \diff{\tau}\\
		&\quad -\sum_{\ell = 1}^{s}\int_{t}^{t_n} \dualVell{f_{\ell}^k(\tau) -
			B_{\ell}(\tau)}{v} \varphi(\tau) \diff{\tau} \Big)
	\end{align*}
	for $t \in (t_{n-1},t_n]$, $n \in \{1,\dots,N\}$.
	Applying integration by parts and the fact that the linear and the
	constant interpolations always coincide at the grid points then shows that
	\begin{align*}
		&\int_{0}^{t_n} \dualV{ 
			(\tilde{U}^k)'(\tau)}{v} \varphi(\tau)
		\diff{\tau}\\
		&= \skab[H]{U^k(t_n)}{v}
		\varphi(t) - \skab[H]{U^k(0)}{v} \varphi(0) -
		\int_{0}^{t_n} \skab[H]{\tilde{U}^k(\tau)}{v} \varphi'(\tau)
		\diff{\tau}.
	\end{align*}
	Recall that in Lemma~\ref{lem:ConvSubSequenceS} we have proved that 
	$\tilde{U}^k \weaks U$ in $L^{\infty}(0,T;H)$ as $k \to 0$.   
	Therefore, \eqref{4eq:LimitProof2} and the fact that $(f_{\ell}^k - 
	B_{\ell})_{k>0}$ is a bounded sequence in $L^q(0,T;V_{\ell}^*)$ for every $\ell 
	\in \{1,\dots,s\}$ shows that
	\begin{align*}
		& \ska[H]{U(t)}{v} \varphi(t) - \ska[H]{U(0)}{v} \varphi(0)
		- \int_{0}^{t} \ska[H]{U(\tau)}{v} \varphi'(\tau) \diff{\tau}\\
		& = \ska[H]{y_t}{v} \varphi(t) - \ska[H]{u_0}{v} \varphi(0)
		- \int_{0}^{t} \ska[H]{U(\tau)}{v} \varphi'(\tau) \diff{\tau},
	\end{align*}
	is fulfilled. This implies $U(t) = y_t$ and $U(0) = u_0$ for every $t \in [0,T]$. 
	
	It remains to prove that $B  = A U$ is fulfilled. To this end, we use
	Lemma~\ref{lem:Minity}. Applying Lemma~\ref{lem:ConvSubSequenceS}, it 
	follows that $U_{\ell}^k \weak U$ in $L^p(0, T;V_{\ell})$ as $k \to 0$ for 
	$\ell \in \{1,\dots,s\}$. Further, in \eqref{4eq:LimitProof2}, we have seen that 
	$A_{\ell}^k U_{\ell}^k \weak B_{\ell}$ in $L^{q}(0, T;V_{\ell}^*)$ as $k \to 0$ 
	for every $\ell \in \{1,\dots,s\}$.  Therefore, we still have to verify
	\begin{align*}
		\limsup_{k \to 0} \sum_{\ell = 1}^{s} \int_{0}^{T}
		\dualVell{A_{\ell}^k(t) U_{\ell}^k(t)}{U_{\ell}^k(t)} \diff{t}
		\leq \int_{0}^{T} \dualV{B(t) }{U(t)} \diff{t}
	\end{align*}
	in order to apply Lemma~\ref{lem:Minity}.
	To this end, we test the semidiscrete problem \eqref{2eq:semiDiscret} 
	with $\U_{\ell}^n$ for $\ell \in \{1,\dots,s\}$ and $n\in \{1,\dots,N\}$ to 
	obtain that
	\begin{align*}
		\dualVell{\U_{\ell}^n - \U^{n-1} + s k \A_{\ell}^n
			\U_{\ell}^n}{\U_{\ell}^n} = \dualVell{sk \F_{\ell}^n}{\U_{\ell}^n}.
	\end{align*}
	Summing up the equation form $\ell = 1$ to $s$, dividing by $s$, and
	applying the identity from \eqref{3eq:discreteDeriv}, it follows that
	\begin{align*}
		\frac{1}{2s} \sum_{\ell = 1}^{s} \big( \|\U_{\ell}^n
		\|_H^2 - \| \U^{n-1}\|_H^2\big)
		+ k \sum_{\ell = 1}^{s} \dualVell{\A_{\ell}^n \U_{\ell}^n}{\U_{\ell}^n}
		\leq k \sum_{\ell = 1}^{s} \dualVell{\F_{\ell}^n}{\U_{\ell}^n}.
	\end{align*}
	After another summation for $n = 1$ to $N$ and an application of  and
	\eqref{3eq:HoelderU}, we can rewrite this inequality to
	\begin{align*}
		\frac{1}{2s} \sum_{\ell = 1}^{s} \|\U_{\ell}^N \|_H^2 - \| u_0 \|_H^2
		+ k \sum_{n=1}^{N} \sum_{\ell = 1}^{s} \dualVell{\A_{\ell}^n
			\U_{\ell}^n}{\U_{\ell}^n}
		\leq k \sum_{n=1}^{N} \sum_{\ell = 1}^{s} \dualVell{\F_{\ell}^n}{\U_{\ell}^n},
	\end{align*}
	due to a telescopic sum structure. Inserting the definition of the 
	prolongations from \eqref{4eq:defConstInt} and \eqref{4eq:defLinInt}, we 
	see that
	\begin{align*}
		&\frac{1}{2s} \sum_{\ell = 1}^{s}\big( \|U_{\ell}^k(T) \|_H^2 - \| u_0 \|_H^2 
		\big)
		+\sum_{\ell = 1}^{s} \int_{0}^{T} \dualVell{ A_{\ell}^k(t)
			U_{\ell}^k(t)}{U_{\ell}^k(t)} \diff{t} \\
		&\leq \sum_{\ell = 1}^{s} \int_{0}^{T} \dualVell{
			f_{\ell}(t) }{U_{\ell}^k(t)} \diff{t}.
	\end{align*}
	Together with \eqref{3eq:HoelderU} and the weak lower semicontinuity
	of the norm this yields
	\begin{align*}
		& \limsup_{k \to 0} \sum_{\ell = 1}^{s} \int_{0}^{T}
		\dualVell{A_{\ell}^k(t) U_{\ell}^k(t)}{U_{\ell}^k(t)} \diff{t} \\
		&\leq \limsup_{k \to 0} \Big(\sum_{\ell = 1}^{s} \int_{0}^{T} \dualVell{
			f_{\ell}^k(t)}{U_{\ell}^k(t)} \diff{t}
		- \frac{1}{2s} \sum_{\ell = 1}^{s} \big(\|U_{\ell}^k(T) \|_H^2
		- \| u_0\|_H^2\big)\Big)\\
		&\leq \limsup_{k \to 0} \Big(\sum_{\ell = 1}^{s} \int_{0}^{T} \dualVell{
			f_{\ell}^k(t)}{U_{\ell}^k(t)} \diff{t}
		- \frac{1}{2} \| U^k(T) \|_H^2
		+ \frac{1}{2} \| u_0\|_H^2\big)\Big)\\
		&\leq \sum_{\ell = 1}^{s} \int_{0}^{T} \dualVell{f_{\ell}(t) }{U(t)} \diff{t}
		- \frac{1}{2} \| U(T) \|_H^2 +  \frac{1}{2} \| u_0\|_H^2\\
		&= \int_{0}^{T} \dualV{f(t) }{U(t)} \diff{t} -
		\int_{0}^{T} \dualV{U'(t)}{U(t)} \diff{t}.
	\end{align*}
	Therefore, we have proved that
	\begin{align*}
		\limsup_{k \to 0} \sum_{\ell = 1}^{s} \int_{0}^{T}
		\dualVell{A_{\ell}^k(t) U_{\ell}^k(t)}{U_{\ell}^k(t)} \diff{t}
		\leq \int_{0}^{T} \dualV{B(t) }{U(t)} \diff{t}.
	\end{align*}
	Applying Lemma~\ref{lem:Minity}, this verifies that $B = AU$ is fulfilled in
	$L^q(0,T;V^*)$. Thus, $U$ is a variational solution to the original 
	problem \eqref{2eq:pde}. 
	As this problem has a unique solution $u \in \W$, it follows that $U = u$.
	
	Next, we argue that the original sequence $(U_{\ell}^k)_{k>0}$
	converges weakly to the unique solution $u$ of \eqref{2eq:pde} in
	$L^p(0,T;V_{\ell})$ for every $\ell \in \{1,\dots,s\}$.
	The arguments above show that every converging subsequence of the
	bounded sequence $(U_{\ell}^k)_{k>0}$ has the limit $u$. Applying the
	subsequence principle, see, e.g, \cite[Proposition~10.13]{Zeidler.1986},
	yields that the entire sequence converges to this limit which proves
	\eqref{4eq:convLp}.
	An analogous argumentation shows that
	\eqref{4eq:convH}--\eqref{4eq:convDeriv} hold true for the original
	sequence. To prove \eqref{4eq:convAu}, we recall \eqref{4eq:convDeriv} and 
	the statement of Lemma~\ref{lem:ConvergenceF}. Inserting these two 
	limiting process in \eqref{4eq:pdeSemi} yields \eqref{4eq:convAu}.
\end{proof}

\begin{theorem}\label{thm:limitStrong}
	Let Assumptions~\ref{ass:Aell} and \ref{ass:fell} be fulfilled.
	Then for step sizes $k = \frac{T}{N}$ the sequence $(U^k)_{k>0}$ defined 
	in  \eqref{4eq:defConstInt} fulfills
	\begin{align*}
		U^k(t) \to u(t) \quad \text{in }H
		\quad \text{as } k \to 0, \quad \text{for } t \in [0,T],
	\end{align*}
	where $u$ is the solution \eqref{2eq:pde}. If $\eta$ in 
	Assumption~\ref{ass:A}~(3) is strictly positive then the sequence 
	$(U_{\ell}^k)_{k>0}$ converges strongly to $u$ in $L^p(0,T;V_{\ell})$ for 
	$\ell \in \{1,\dots,s\}$.
\end{theorem}

\begin{proof}
	For the analysis we split up the terms as follows
	\begin{align*}
		&\frac{1}{s} \sum_{\ell = 1}^{s} \| u(t) - U_{\ell}^k(t) \|_H^2 \\
		&+ 2 \sum_{\ell = 1}^{s} \int_{0}^{t} \dualVell{A_{\ell}^k(\tau)u(\tau)
			- A_{\ell}^k(\tau) U_{\ell}^k(\tau) }{ u(\tau) - U_{\ell}^k(\tau)} 
			\diff{\tau}\\
		&= X_1^k(t) + X_2^k(t) + X_3^k(t)
	\end{align*}
	with
	\begin{align*}
		X_1^k(t)
		&= \| u(t) \|_H^2 + 2 \sum_{\ell = 1}^{s} \int_{0}^{t} 
		\dualVell{A_{\ell}^k(\tau)
			u(\tau)}{u(\tau)} \diff{\tau},\\
		X_2^k(t)
		&= - \frac{2}{s} \sum_{\ell = 1}^{s} \ska[H]{ u(t) }{ U_{\ell}^k(t) } -
		2\sum_{\ell = 1}^{s} \int_{0}^{t}
		\dualVell{A_{\ell}^k(\tau) u(\tau)}{U_{\ell}^k(\tau)} \diff{\tau}\\
		&\quad- 2\sum_{\ell = 1}^{s} \int_{0}^{t} \dualVell{A_{\ell}^k(\tau)
			U_{\ell}^k(\tau) }{ u(\tau)}
		\diff{\tau},\\
		X_3^k(t)
		&= \frac{1}{s} \sum_{\ell = 1}^{s} \| U_{\ell}^k(t) \|_H^2
		+ 2 \sum_{\ell = 1}^{s} \int_{0}^{t} \dualVell{A_{\ell}^k(\tau)
			U_{\ell}^k(\tau) }{U_{\ell}^k(\tau)} \diff{\tau}.
	\end{align*}
	We analyze $X_1^k$, $X_2^k$ and $X_3^k$ separately.
	For $X_1^k$ we apply Lemma~\ref{lem:ConvergenceA} and obtain
	\begin{align*}
		\lim_{k \to 0} X_1^k(t)
		= \| u(t) \|_H^2 + 2 \int_{0}^{t} \dualV{A(\tau) u(\tau)}{u(\tau)}
		\diff{\tau}.
	\end{align*}
	We use Lemma~\ref{lem:ConvergenceA}, \eqref{4eq:convLp}, 
	\eqref{4eq:convH} and \eqref{4eq:convAu} as well as the definition of 
	$U^k$ from \eqref{4eq:defConstInt} to see
	\begin{align*}
		\lim_{k \to 0} X_2^k(t)
		&= \lim_{k \to 0} \Big(- 2\ska[H]{ u(t) }{ U^k(t) } -
		2\sum_{\ell = 1}^{s} \int_{0}^{t}
		\dualVell{A_{\ell}^k(\tau) u(\tau)}{U_{\ell}^k(\tau)} \diff{\tau}\\
		&\quad- 2\sum_{\ell = 1}^{s} \int_{0}^{t} \dualVell{A_{\ell}^k(\tau)
			U_{\ell}^k(\tau) }{ u(\tau)}
		\diff{\tau}\Big)\\
		&= - 2\| u(t) \|_H^2 - 4 \int_{0}^{t} \dualV{A(\tau) u(\tau)}{u(\tau)}
		\diff{\tau}.
	\end{align*}
	The convergence of $(X_3^k(t))_{k >0}$ needs somewhat more
	attention.
	Here, we assume that $t \in (t_{n-1},t_n]$, $n \in \{1,\dots,N\}$, and obtain
	\begin{align*}
		X_3^k(t)
		&= \frac{1}{s} \sum_{\ell =1 }^{s} \| U_{\ell}^k(t) \|_H^2
		+ 2 \sum_{\ell =1 }^{s} \int_{0}^{t} \dualVell{A_{\ell}^k(\tau)
			U_{\ell}^k(\tau) }{U_{\ell}^k(\tau)} \diff{\tau}\\
		&\leq \frac{1}{s} \sum_{\ell =1 }^{s} \| \U_{\ell}^n \|_H^2
		+ 2k \sum_{i =1 }^{n} \sum_{\ell =1 }^{s}
		\dualVell{\A_{\ell}^i \U_{\ell}^i - \F_{\ell}^i}{\U_{\ell}^i}\\
		&\quad + 2\sum_{\ell =1 }^{s} \int_{0}^{t_n}
		\dualVell{f_{\ell}^k(\tau)}{U_{\ell}^k(\tau)} \diff{\tau}\\
		&\quad - 2 \sum_{\ell =1 }^{s} \int_{t}^{t_n} \dualVell{A_{\ell}^k(\tau)
			U_{\ell}^k(\tau) }{U_{\ell}^k(\tau)} \diff{\tau}.
	\end{align*}
	Inserting \eqref{3eq:HoelderU} and the identity \eqref{3eq:discreteDeriv} 
	as well as applying a telescopic sum argument, it follows that
	\begin{align*}
		&2k \sum_{i=1}^{n}\sum_{\ell =1 }^{s} \dualVell{\A_{\ell}^i \U_{\ell}^i -
			\F_{\ell}^i}{\U_{\ell}^i}\\
		&= - \frac{2}{s} \sum_{i=1}^{n} \sum_{\ell =1 }^{s} \ska[H]{\U_{\ell}^i -
			\U^{i-1}}{\U_{\ell}^i}
		\leq- \frac{1}{s} \sum_{\ell =1 }^{s} \|\U_{\ell}^n\|_H^2 + \|u_0\|_H^2.
	\end{align*}
	Assumption~\ref{ass:A}~(5) yields the bound
	\begin{align*}
		- 2 \sum_{\ell =1 }^{s} \int_{t}^{t_n} \dualVell{A_{\ell}^k(\tau)
			U_{\ell}^k(\tau) }{U_{\ell}^k(\tau)} \diff{\tau}
		\leq  2k s \lambda.
	\end{align*}
	Therefore, we obtain that
	\begin{align*}
		X_3^k(t)
		&\leq \| u_0 \|_H^2
		+ 2 \sum_{\ell =1 }^{s} \int_{0}^{t_n}
		\dualVell{f_{\ell}^k(\tau) } {U_{\ell}^k(\tau)} \diff{\tau} + 2k s \lambda.
	\end{align*}
	Due to Lemma~\ref{lem:ConvergenceF} and \eqref{4eq:convLp}, it follows that
	\begin{align*}
		\limsup_{k \to 0} X_3^k(t)
		\leq \| u_0 \|_H^2
		+ 2 \int_{0}^{t} \dualV{f(\tau) } {u(\tau)} \diff{\tau}.
	\end{align*}
	Thus, we have proved that
	\begin{align*}
		&\limsup_{k \to 0} \Big( \frac{1}{s} \sum_{\ell = 1}^{s} \| u(t) - U_{\ell}^k(t) 		
		\|_H^2 \\
    &\quad + 2\sum_{\ell = 1}^{s} \int_{0}^{t} \dualVell{A_{\ell}^k(\tau)u(\tau) -
			A_{\ell}^k(\tau) U_{\ell}^k(\tau) }{ u(\tau) - U_{\ell}^k(\tau)} \diff{\tau} \Big)\\
		&\leq  - \| u(t) \|_H^2 + \| u_0 \|_H^2
		+ 2 \int_{0}^{t} \dualV{f(\tau) - A(\tau) u(\tau)} {u(\tau)} \diff{\tau}\\
		&=  - \| u(t) \|_H^2 + \| u_0 \|_H^2
		+ 2 \int_{0}^{t} \dualV{u'(\tau)} {u(\tau)} \diff{\tau}\\
		&=  - \| u(t) \|_H^2 + \| u_0 \|_H^2
		+  \int_{0}^{t} \frac{\diff{}}{\diff{t}} \| u(\tau) \|_H^2 \diff{\tau}
		= 0.
	\end{align*}
	The monotonicity condition from Assumption~\ref{ass:A}~(3) and 
	\eqref{3eq:HoelderU} then imply that
	\begin{align*}
		&\| u(t) - U^k(t) \|_H^2\\
		&\leq \frac{1}{s} \sum_{\ell = 1}^{s} \| u(t) - U_{\ell}^k(t) \|_H^2
		+ 2\eta  \sum_{\ell = 1}^{s} \int_{0}^{t} |u(\tau) - 
		U_{\ell}^k(\tau)|_{V_{\ell}}^p \diff{\tau}\\
    &\leq \Big( \frac{1}{s} \sum_{\ell = 1}^{s} \| u(t) - U_{\ell}^k(t) \|_H^2 \\
    &\quad + 2\sum_{\ell = 1}^{s} \int_{0}^{t} \dualVell{A_{\ell}^k(\tau)u(\tau) -
      A_{\ell}^k(\tau) U_{\ell}^k(\tau) }{ u(\tau) - U_{\ell}^k(\tau)} \diff{\tau} \Big)
		\to 0
	\end{align*}
	as $k \to 0$ for every $t \in [0,T]$. This proves that $U^k(t) \to u(t)$ in 
	$H$ as $k \to0$ for every $t \in [0,T]$.
	Assuming that $\eta$ in Assumption~\ref{ass:A}~(3) is strictly positive and 
	applying the norm bound from Assumption~\ref{ass:spaces},  
	it follows that
	\begin{align*}
		\| u - U_{\ell}^k \|_{L^p(0,T;V_{\ell})}
		&\leq c_{V_{\ell}} \Big(\int_{0}^{T} \big(\|u(t) - U_{\ell}^k(t)\|_{H}
		+ |u(t) - U_{\ell}^k(t)|_{V_{\ell}}\big)^p \diff{t}\Big)^{\frac{1}{p}}
		\to 0
	\end{align*}
	as $k \to 0$ for every $\ell \in \{1,\dots,s\}$.
\end{proof}

\begin{remark}\label{remark:etaNonZero}
	If $\eta$ from Assumption~\ref{ass:A}~(3) is only strictly positive for 
	some of the operator families $\{A_{\ell}(t)\}_{t\in[0,T]}$, $\ell \in \{1,\dots,s\}$, 
	then one can see that in these particular spaces we have $U_{\ell}^k \to u$ in 
	$L^p(0,T;V_{\ell})$.
\end{remark}

\end{document}